\documentclass{amsart}
\pdfoutput=1

\usepackage[mathscr]{eucal}
\usepackage{amssymb}
\usepackage[usenames,dvipsnames]{xcolor} 
\usepackage[normalem]{ulem}
\usepackage{amsthm}
\usepackage{mathtools}
\usepackage{bbold}
\usepackage{mathdots}
\usepackage{float}
\usepackage{enumerate}
\usepackage{amsmath}
\usepackage{wasysym}
\usepackage[all]{xy}
\SelectTips{cm}{}
\newdir{ >}{{}*!/-10pt/\dir{>}}

\usepackage[unicode]{hyperref} 

\hypersetup{colorlinks=true,citecolor=Brown,urlcolor=Brown,linkcolor=Brown}

\usepackage{microtype}
\usepackage[nameinlink,capitalise,noabbrev]{cleveref}

\numberwithin{equation}{section}
\swapnumbers 

\newtheorem{Thm}[equation]{Theorem}
\newtheorem*{Thm*}{Theorem}
\newtheorem{Prop}[equation]{Proposition}
\newtheorem{Lem}[equation]{Lemma}
\newtheorem{Cor}[equation]{Corollary}

\theoremstyle{remark}
\newtheorem{Def}[equation]{Definition}
\newtheorem{Ter}[equation]{Terminology}
\newtheorem{Not}[equation]{Notation}
\newtheorem{Exa}[equation]{Example}
\newtheorem{Ack}[equation]{Acknowledgements}

\newtheorem{Rem}[equation]{Remark}

\newtheorem{Rec}[equation]{Recollection}

\newcommand{\nc}{\newcommand}
\nc{\dmo}{\DeclareMathOperator}

\nc{\Beren}[1]{{\color{MidnightBlue}#1}}
\nc{\Bout}[1]{\Beren{\sout{#1}}}
\nc{\Paul}[1]{{\color{ForestGreen}#1}}
\nc{\Pout}[1]{\Beren{\sout{#1}}}

\usepackage{todonotes}

\dmo{\cone}{cone}
\dmo{\Der}{D}
\dmo{\hocolim}{hocolim}
\dmo{\holim}{holim}
\dmo{\Hom}{Hom}
\dmo{\id}{id}
\dmo{\Id}{Id}
\dmo{\incl}{incl}
\dmo{\Ker}{Ker}
\dmo{\kos}{kos}
\dmo{\Loc}{Loc}
\dmo{\Mat}{M}
\dmo{\OO}{\mathcal{O}}
\dmo{\opname}{op}
\dmo{\proj}{proj}
\dmo{\rmH}{H}
\dmo{\Spc}{Spc}
\dmo{\Spec}{Spec}
\dmo{\supp}{supp}
\dmo{\thick}{thick}

\nc{\adj}{\dashv}
\nc{\bbN}{\mathbb{N}}
\nc{\bbQ}{\mathbb{Q}}
\nc{\bbZ}{\mathbb{Z}}
\nc{\cat}[1]{\mathscr{#1}}
\nc{\cS}{\cat{S}}
\nc{\cT}{\cat{T}}
\nc{\cU}{\cat{U}}
\nc{\Dperf}{\Der^{\mathrm{perf}}}
\nc{\DYR}{\Der_Y\hspace{-0.15ex}(R)}
\nc{\ee}{\mathbb{e}}
\nc{\eg}{{\sl e.\,g.}}
\nc{\eY}{\ee_Y}
\nc{\ff}{\mathbb{f}}
\nc{\fY}{\ff_Y}
\nc{\gm}{\mathfrak{m}}
\nc{\Homcat}[1]{\Hom_{\cat #1}}
\nc{\hook}{\hookrightarrow}
\nc{\hTY}{{\cT}^{\wedge}_Y}
\nc{\ideal}[1]{\langle #1\rangle}
\nc{\ie}{{\sl i.e.}, }
\nc{\In}{I^{(n)}}
\nc{\Inm}{I^{(n-1)}}
\nc{\into}{\mathop{\rightarrowtail}}
\nc{\inv}{^{-1}}
\nc{\ion}{\iota_{n}}
\nc{\ionm}{\iota_{n-1}}
\nc{\isoto}{\mathop{\overset{\sim}\to}}
\nc{\kn}{k^{(n)}}
\nc{\knm}{k^{(n-1)}}
\nc{\knp}{k^{(n+1)}}
\nc{\Lotimes}{\otimes}
\nc{\Mid}{\,\big|\,}
\nc{\onto}{\mathop{\twoheadrightarrow}}
\nc{\op}{^{\opname}}
\nc{\potimes}[1]{^{\otimes #1}}
\nc{\quadtext}[1]{\quad\textrm{#1}\quad}
\nc{\restr}[1]{_{|_{\scriptstyle #1}}}
\nc{\sbull}{{\scriptscriptstyle\bullet}}
\nc{\SET}[2]{\big\{\,#1\Mid#2\,\big\}}
\nc{\sminus}{\!\smallsetminus\!}
\nc{\SpcS}{\Spc(\cat S^c)}
\nc{\SpcT}{\Spc(\cat T^c)}
\nc{\SYp}{\cS^{\wedge}_{Y'}}
\nc{\then}{\Rightarrow}
\nc{\unit}{\mathbb{1}}
\nc{\unitS}{\unit_{\cat S}}
\nc{\xto}[1]{\xrightarrow{#1}}
\nc{\Ycomplete}{$Y$\hspace{-0.4ex}-complete}
\nc{\Ycompletion}{$Y$\hspace{-0.4ex}-completion}
\nc{\Ytorsion}{$Y$\hspace{-0.4ex}-torsion}

\renewcommand{\ln}{\ell^{(n)}}

\Crefname{Thm}{Theorem}{Theorems}
\Crefname{Prop}{Proposition}{Propositions}
\Crefname{thmx}{Theorem}{Theorems}

\title[Perfect complexes and completion]{Perfect complexes and completion}

\author[P.\ Balmer]{Paul Balmer}
\address{UCLA Mathematics Department, Los Angeles, CA 90095-1555, USA}
\email{balmer@math.ucla.edu}

\author[B.\ Sanders]{Beren Sanders}
\address{Mathematics Department, UC Santa Cruz, 95064 CA, USA}
\email{beren@ucsc.edu}
\urladdr{http://people.ucsc.edu/$\sim$beren/}

\date{2024 November 21}



\thanks{The first author was supported by NSF grant DMS-215375.}

\begin{document}


\maketitle

\begin{abstract}
\vskip-\baselineskip
\vskip-\baselineskip
Let~$\hat{R}_{}$ be the $I$-adic completion of a commutative ring~$R$ with respect to a finitely generated ideal~$I$.
We give a necessary and sufficient criterion for the category of perfect complexes over~$\hat{R}_{}$ to be equivalent to the subcategory of dualizable objects in the derived category of $I$-complete complexes of~$R$-modules.
Our criterion is always satisfied when $R$ is noetherian.
When specialized to~$R$ local and noetherian and to~$I$ the maximal ideal, our theorem recovers a recent result of Benson, Iyengar, Krause and Pevtsova.
\end{abstract}

{
\hypersetup{linkcolor=black}
\vskip\baselineskip\vskip\baselineskip
\tableofcontents
}

\vskip-\baselineskip
\vskip-\baselineskip
\vskip-\baselineskip
\section{Introduction}


Let $R$ be a commutative ring
and let $I\subset R$ be a finitely generated ideal.
We write $Y\coloneqq V(I)\subseteq \Spec(R)$ for the corresponding closed subset with quasi-compact complement.
It is well-known that the $I$-adic completion~$\smash{\hat{R}_I}\coloneqq\lim_{n}R/I^n$ only depends on~$Y$. We sometimes write~$\smash{\hat{R}_Y}$, or just~$\smash{\hat{R}}$, for this ring.

It is legitimate to ask whether the derived category~$\Der(\hat{R}_{{}})$ of the completion can be recovered from the derived category of~$R$ in purely tensor-triangular terms.
The question is tantalizing since the term `completion' also refers to a standard construction in stable homotopy theory.
Let us remind the reader.

Suppose that $\cT$ is a `big' tensor-triangulated category.
The exact hypothesis will play an interesting role, so we give some details.
We want the category $\cT$ to be compactly generated in the sense of Neeman~\cite{Neeman01} and to admit a good notion of `small' object, namely, we want the compact objects~$\cT^c$ and the dualizable objects~$\cT^d$ to coincide: $\cT^c=\cT^d$.
There is a profusion of such `big' tt-categories in mathematics, see~\cite{BalmerFavi11}
or~\cite{HoveyPalmieriStrickland97}.
The derived category~$\cT=\Der(R)$ is an example; its small objects $\cT^c=\cT^d=\smash{\Dperf(R)}$ are precisely the \emph{perfect complexes}, \ie the bounded complexes of finitely generated projective $R$-modules.

For any `big' tt-category~$\cT$, consider the spectrum $\SpcT$ of the small part~$\cT^c$.
Choose a closed subset $Y\subseteq \Spc(\cT^c)$ with quasi-compact complement and let $\cT^c_Y=\SET{c\in\cT^c}{\supp(c)\subseteq Y}$ be the tt-ideal of small objects supported on~$Y$. In the literature (see Greenlees~\cite[\S\,2]{Greenlees01}), the double right-orthogonal of~$\cT^c_Y$ in~$\cT$
\begin{equation}
\label{eq:TcY-perp-perp}%
\hTY:=((\cT^c_Y)^{\perp})^{\perp}
\end{equation}
is called the subcategory of \emph{{\Ycomplete} objects} in~$\cT$.
See more in~\Cref{Rem:TY=hTY}.
The tt-category~$\hTY$ comes with a tt-functor called \emph{completion}
\begin{equation}
\label{eq:compl}%
(-)^\wedge_Y\colon \ \cT\to \hTY
\end{equation}
which is simply a localization, left adjoint to the inclusion~$\hTY\hookrightarrow\cT$.

Unfortunately, if we plug $\cT:=\Der(R)$ into this abstract theory, using the identification $\SpcT\cong\Spec(R)$, it essentially never happens that $\hTY$ is the derived category of~$\hat{R}_{Y}$, except in the trivial case where the ring splits~$R\simeq R_1\times R_2$ and~$I=R_2$ (so $\hat{R}_{I}=R_1$).
It even fails for the $p$-adics: $\Der(\bbZ)^\wedge_{(p)}\not\simeq\Der(\hat\bbZ_p)$.
So `derived category' and `completion' do not commute.
And it gets worse. Except in the split case, $\hTY$ is not the derived category of \emph{any} commutative ring whatsoever.

The reason for this is well-known: The tt-category~$\hTY$ is not a legitimate `big' tt-category.
It is generated by compact-dualizable objects but its $\otimes$-unit, although obviously dualizable, is only compact in split cases (\Cref{Rem:trivial-tt}).
In other words, for every `big' tt-category~$\cT$ and every~$Y$
the inclusion
\begin{equation}
\label{eq:TY-cd}%
(\hTY)^c \hook (\hTY)^d
\end{equation}
is typically a proper inclusion.
This distinction, between \emph{compact} and \emph{dualizable}, only makes sense in \emph{tensor}-triangular geometry and is difficult to extract from the mere triangulated structure.
When the two categories in~\eqref{eq:TY-cd} are different, the larger one is better. Indeed, only~$(\hTY)^d$ is an actual tt-category and only~$(\hTY)^d$ receives the original~${\cT^c=\cT^d}$ under the completion tt-functor~$(-)^\wedge_Y$ in~\eqref{eq:compl}.

A significant observation of Benson--Iyengar--Krause--Pevtsova~\cite[\S\,4]{BensonIyengarKrausePevtsova23} is that although we do not recover~$\Der(\hat{R}_{Y})$ as the {\Ycompletion} of~$\cT=\Der(R)$, we can still hope to recover its small objects.
They prove that when $(R,\gm)$ is a local and noetherian ring and $Y=\{\gm\}$, the category~$(\hTY)^d$, whose praise we sang above, recovers the derived category of \emph{perfect} complexes over the $\gm$-adic completion~$\smash{\hat{R}_{\gm}}$.
Interestingly, BIKP did not conjecture that their result would hold for any commutative ring~$R$ and any finitely generated ideal~$I$.
And indeed it fails in that generality, as we shall see in \Cref{Rem:no}.
On the other hand, we shall prove that the noetherian result is not restricted to local rings but is rather a global fact:
\begin{Thm}[{\Cref{Cor:main-noeth}}]
\label{Thm:main-noeth-intro}%
Suppose that $R$ is noetherian. Then there is a canonical equivalence of tt-categories between the derived category~$\Dperf(\hat{R}_{Y})$ of perfect complexes over the completion and the category~$(\Der(R)^{\wedge}_Y)^d$ of dualizable objects in the tt-category of {\Ycomplete} complexes of~$R$-modules, which makes the following diagram commute
\begin{equation}
\label{eq:equivalence-comp}%
\vcenter{
\xymatrix@R=.9em@C=1em{
&& \Dperf(R) \ar@/_1em/[ld]_-{(-)^\wedge_Y} \ar@/^1em/[rd]^-{\hat{R}_{{}}\otimes_R-}
\\
& (\Der(R)^{\wedge}_Y)^d \ar[rr]^-{\cong}
&& \Dperf(\hat{R}_{Y}).
}}
\end{equation}
\end{Thm}

This statement is an easy consequence of our results concerning possibly non-noetherian rings.
In that generality, the above formulation fails, as already mentioned, and we need to add a new condition.
We want to express that condition in `concrete' terms and will do so by means of Koszul complexes.

But before wheeling in Koszul complexes, let us point out another problem that we shall face in the non-noetherian setting:
The two notions of completion discussed so far can yield different ring objects, that is, the tt-completion~$\hat{\unit}_{Y}$ and the classical completion~$\hat{R}_{Y}$ might not be isomorphic in~$\cT=\Der(R)$.
We shall see that this problem has the `same' solution as the first one.
This second result can even be formulated at the level of abstract `big' tt-categories, not just for $\Der(R)$; see \Cref{Thm:tt-compl}.

So let us now come to our criterion.
For a sequence $\underline{s}=(s_1,\ldots,s_r)$ of generators of the ideal~$I$, consider the
usual Koszul complex~$\kos_R(\underline{s})=\otimes_{i=1}^{r}\cone(s_i\colon \unit\to \unit)$.
In degree zero, the homology of this complex is simply $R/I$ and therefore does not change if we replace~$R$ by~$\hat{R}_{}$.
We say that the sequence~$\underline{s}$ is \emph{Koszul-complete} if this holds not only in degree zero but in all degrees:
$\rmH_\sbull(\kos_R(\underline{s}))\cong\rmH_\sbull(\kos_{\hat{R}_{}}(\underline{s}))$; see~\Cref{Def:I-kos}.
This condition always holds when~$R$ is noetherian, as we verify in~\Cref{prop:noeth-koszul}.
We can now state our main result.
\begin{Thm}[{\Cref{Thm:main}}]
\label{Thm:main-intro}%
Let $R$ be a commutative ring and let $\underline{s}=(s_1,\ldots,s_r)$ be a sequence of elements with $Y\coloneqq V(s_1,\ldots,s_r)$.
Then the following are equivalent:
\begin{enumerate}[\rm(i)]
\item
The sequence~$\underline{s}$ is Koszul-complete: $\rmH_\sbull(\kos_R(\underline{s}))\cong\rmH_\sbull(\kos_{\hat{R}_{}}(\underline{s}))$.
\smallbreak
\item
There is an equivalence of tt-categories between $\Dperf(\hat{R}_{Y})$ and $(\Der(R)^\wedge_Y)^d$ which makes the diagram in~\eqref{eq:equivalence-comp} commute.
\smallbreak
\item
There exists an isomorphism of ring objects~$\hat{\unit}_Y\cong\hat{R}_{Y}$ in~$\Der(R)$. In other words, tt-completion of the unit along~$Y$ recovers classical ring completion.
\end{enumerate}
\end{Thm}

As we shall see, our proof is more involved than the proof given in~\cite{BensonIyengarKrausePevtsova23} in the special case where $R$ is local and noetherian and~$Y=\{\gm\}$ is the closed point.
As is already apparent from its title, the article \cite{BensonIyengarKrausePevtsova23} adopts a very local focus;
noetherianity of~$R$ is used at every step and the residue field $R/\gm$ is used for dimension arguments.
None of those tools are available to us. We give a completely independent proof that does not assume their special case.
This being said, we owe to~\cite{BensonIyengarKrausePevtsova23} the insight that such statements could even be true.

Let us return to the original question of recovering~$\hat{\cT}\coloneqq\Der(\hat{R}_{})$ from~$\cT=\Der(R)$.
In the Koszul-complete case (\eg\ if~$R$ is noetherian), we have recovered the small objects~$\hat{\cT}^c$ from~$\cT$.
Those readers who want to think of~$\cT$ and its acolytes~$\hTY$ and~$(\hTY)^d$ as homotopy categories of some $\infty$-categories can now construct $\smash{\hat{\cT}}$ as the \emph{Ind-completion} of~$\smash{\hat{\cT}^c}$.
In the non-Koszul-complete case, \ie when the two rings $\hat{\unit}_Y$ and~$\hat{R}_{Y}$ are different, we need to make a choice.
Since $I$-adic completion has long been a source of headaches outside the noetherian world, now might be the time to let go of~$\hat{R}_{Y}$, to opt for the better-behaved $(\hTY)^d$, and to simply Ind-complete the latter to define~$\smash{\hat{\cT}}$
in full generality, as recently proposed in Naumann--Pol--Ramzi~\cite{NaumannPolRamzi24}.
We comment further on this topic in the final \Cref{Rem:final}.

The outline of the paper is relatively straightforward. In \Cref{sec:tt-completion}, we review the abstract notion of completion in `big' tt-categories~$\cT$.
In \Cref{sec:comm-alg}, we specialize the discussion to $\cT=\Der(R)$ and we contrast tt-completion with classical $I$-adic completion.
The technical heart of the paper beats in \Cref{sec:dualizables-in-TY} where we prove that when $R\cong\hat{R}_{Y}$ is classically complete, an object of~$\Der(R)^\wedge_Y$ is dualizable in that category if and only if it is a perfect complex. This is \Cref{Cor:dualizable-complete}, which holds unconditionally and may be of independent interest.
In the final \Cref{sec:main}, we prove \Cref{Thm:main-intro} and its corollaries.

\begin{Ack}
The authors would like to thank Greg Stevenson for numerous discussions about completion, over many years.
\end{Ack}


\section{Tensor-triangular completion}
\label{sec:tt-completion}%


In this section, $\cT$ stands for a rigidly-compactly generated `big' tensor-triangulated category, as recalled in the introduction.
See details in \cite{BalmerFavi11} if necessary, or in~\cite{HoveyPalmieriStrickland97} where such~$\cT$ are called `unital algebraic stable homotopy categories'.
We denote the internal hom in $\cT$ by $[-,-]\colon \cT{\op}\times\cT\to \cT$.

\begin{Rec}
\label{Rec:idempotents}%
Let $Y\subseteq\SpcT$ be a Thomason subset of the tt-spectrum, for instance a closed subset with quasi-compact open complement.
The choice of the Thomason subset~$Y$ is equivalent to the choice of the thick $\otimes$-ideal of~$\cT^c$:
\[
\cT^c_Y=\SET{ x \in \cat T^c}{\supp(x)\subseteq Y\textrm{ in }\SpcT}.
\]
A central player in this article is the (smashing) localizing $\otimes$-ideal
\[
\cT_Y:=\Loc(\cT^c_Y)
\]
generated by~$\cT^c_Y$ in~$\cT$. This tt-category $\cT_Y$ of \emph{objects of~$\cT$ supported on~$Y$} is sometimes called the \emph{torsion} part.
(See more on its tensor-structure in \Cref{Rem:TY=hTY}.)
The compact objects in~$\cT_Y$ are the above $\cT_Y^c$ by \cite{Neeman92b}; removing any ambiguity,~$(\cT_Y)^c=(\cT^c)_Y=:\cT^c_Y$.
The category $\cT$ admits a recollement with respect to $\cT_Y$ and its orthogonal~$(\cT_Y)^\perp=\SET{u\in\cT}{\Hom_{\cT}(t,u)=0,\,\forall\,t\in\cT_Y}$.
As explained in~\cite{BalmerFavi11}, the functors that appear in this recollement can all be described in terms of two objects~$\eY$ and~$\fY$ that fit in the so-called \emph{idempotent triangle}
\begin{equation}\label{eq:idemp-triangle}%
\eY\to \unit \to \fY\to \Sigma \eY\,.
\end{equation}
The latter is uniquely characterized by $\eY\in\cT_Y$ and~$\fY\in(\cT_Y)^\perp$.
We have $\eY\otimes\fY=0$, forcing $\eY\potimes{2}\cong\eY$ and $\fY\cong\fY\potimes{2}$.
The recollement is then given by:
\begin{equation}
\label{eq:recollement}%
\vcenter{\xymatrix@R=2em{
\cT_Y =\eY\otimes\cT \ar@{ >->}@<-2em>[d]_-{\incl} \ar@{ >->}@<2em>[d]^-{[\eY,-]}
\\
\cT \ar@{->>}[u]|{\vphantom{I_J}\eY\otimes-} \ar@{->>}@<-2em>[d]_(.45){\fY\otimes-} \ar@{->>}@<2em>[d]^-{[\fY,-]}
\\
(\cT_Y)^\perp = \fY\otimes\cT \ar@{ >->}[u]|(.55){\vphantom{I_J}\incl}
}}
\end{equation}
\end{Rec}

\begin{Rec}
\label{Rec:Bousfield}%
The Bousfield--Neeman theory of localization of triangulated categories in terms of orthogonal subcategories can be found in~\cite[Chapter~9]{Neeman01}.
The three vertical sequences
\[
\cS\,\overset{F}{\into} \,\cT\,\overset{G}{\onto} \,\cU
\]
that appear in~\eqref{eq:recollement} are `exact' in the sense that~$G$ identifies $\cU$ with the Verdier quotient~$\cT/F(\cS)$.
When $F$ has a right (respectively, left) adjoint then so does~$G$, and the quotient $\cT/F(\cS)$ becomes equivalent to the subcategory~$\cS^{\perp}$
(respectively, ${}^\perp\cS$)
of~$\cT$ via that adjoint.

In particular, for the middle sequence $(\cT_Y)^\perp\into \cT\onto \cT_Y$ in~\eqref{eq:recollement}, since the inclusion $(\cT_Y)^\perp\into \cT$ has adjoints on both sides, the Verdier quotient $\cT/\cT_Y^\perp$ identifies both with~${}^\perp((\cT_Y)^\perp)=\cT_Y$ and with~$((\cT_Y)^\perp)^\perp$. The former are the {\Ytorsion} objects and the latter are the {\Ycomplete} ones.
This equivalence $\cT_Y\cong(\cT_Y)^{\perp\perp}$ goes back to Dwyer--Greenlees~\cite{DwyerGreenlees02} in the case of derived categories.
\end{Rec}

\begin{Def}
\label{Def:tt-completion}%
Following the literature, we define $\hTY:=(\cT_Y)^{\perp\perp}$ and call it the subcategory of \emph{{\Ycomplete}} objects (or \emph{tt-complete along~$Y$}).
It is sometimes denoted~$\cT_{\textrm{comp}}$.
Since $(\cT_Y)^\perp=(\cT_Y^c)^\perp$, this definition agrees with~\eqref{eq:TcY-perp-perp}.
As explained above, $\hTY$ is equivalent to the Verdier quotient $\cT\big/(\cT_Y)^\perp$ of~$\cT$ by the localizing $\otimes$-ideal~$(\cT_Y)^\perp$ and in particular $\hTY$ inherits from~$\cT$ the structure of a tensor-triangulated category making the localization functor~$\cT\onto \hTY$ into a tt-functor.
\end{Def}

Let us unpack \Cref{Rec:Bousfield} explicitly in our setting.
\begin{Rem}
\label{Rem:TY=hTY}%
The generalized Dwyer--Greenlees equivalence $\cT_Y\cong(\cT_Y)^{\perp\perp}$ is given here by $[\eY,-]\colon\cT_Y\isoto (\cT_Y)^{\perp\perp}$ with inverse given by~$\eY\otimes-$.
In other words, the top part of~\eqref{eq:recollement} can be written in either of the following equivalent forms:
\begin{equation}
\label{eq:TY-vs-hTY}%
\vcenter{\xymatrix@C=10em{
\cT_Y =\eY\otimes\cT \ \ar@{ >->}@/_2em/[d]_-{\incl} \ar@{ >->}@/^2em/[d]^-{[\eY,-]} \ar@/^.5em/[r]^-{[\eY,-]} \ar@{}[r]|-{\cong}
&
\ (\cT_Y)^{\perp\perp} =[\eY,\cT] \ar@{ >->}@/_2em/[d]_-{\eY\otimes-} \ar@{ >->}@/^2em/[d]^-{\incl} \ar@/^.5em/[l]^-{\eY\otimes-}
\\
\quad \cT \quad \ar@{->>}[u]|{\vphantom{I_J}\eY\otimes-} \ar@{=}[r]
&
\quad \cT \quad \ar@{->>}[u]|{\vphantom{I_J}[\eY,-]}
\\
(\cT_Y)^\perp \ar@{ >->}[u]|-{\incl} \ar@{=}[r]
& (\cT_Y)^\perp \ar@{ >->}[u]|-{\incl}
}}
\end{equation}
and the localization~$\cT\onto \cT/(\cT_Y)^\perp$ mentioned in~\Cref{Def:tt-completion} is simply the tt-functor~$[\eY,-]\colon \cT\onto (\cT_Y)^{\perp\perp}$ on the right-hand side of~\eqref{eq:TY-vs-hTY}.
It can equivalently be realized via the functor~$\eY\otimes-\colon \cT\onto \cT_Y$ on the left-hand side of~\eqref{eq:TY-vs-hTY}.
Note that the composite~$\cT\onto \cT/(\cT_Y)^\perp\into\cT$ of the localization followed by its \emph{right} adjoint will always be~$[\eY,-]\colon \cT\to \cT$ in both formulations!
This is one reason why people prefer the latter as the concrete realization of the {\Ycompletion} of~\Cref{Def:tt-completion}:
\begin{equation}
\label{eq:tt-completion}%
\vcenter{\xymatrix@C=1em@R=.2em{
(-)^\wedge_Y:=[\eY,-] \ar@{}[r]|-{\colon}
& \cT\quad \ar@{->>}[rr]
&& \quad \hTY:=(\cT_Y)^{\perp\perp}
\\
& t \quad \ar@{|->}[rr]
&& \quad \hat{t}_Y:=[\eY,t].
}}
\end{equation}

Another reason to use $[\eY,-]$ is that if we write the object $\eY$ as a sequential homotopy colimit of objects~$k_n$ in~$\cT_Y^c$ then $[\eY,t]$ is a homotopy \emph{limit} of the objects $[k_n,t]\cong k_n^\vee\otimes t$ in~$\cT_Y$.
(See for instance \Cref{Cor:[et]}.) This gives a good `completion vibe' to~$(-)^\wedge_Y$: The fact that $[\eY,t]$ is a limit of the {\Ytorsion} objects~$k_n^\vee\otimes t$ is reminiscent of the way $I$-adic completion~$\hat{M}$ is the limit of the $I$-primary torsion objects~$M/I^n$.

For the torsion part, the localization functor~$\cT\onto \cT_Y$ is simply~$\eY\otimes-$ and the tensor in~$\cT_Y$ is simply the ambient one in~$\cT$ (making sure to use the correct unit:~$\eY$).
On the other hand, the tensor in~$(\cT_Y)^{\perp\perp}$ is a \emph{completed tensor}, meaning that one needs to apply $[\eY,-]$ after tensoring {\Ycomplete} objects in~$\cT$.
This completed tensor is nothing remarkable though. The equivalence~$[\eY,-]\colon\cT_Y\isoto (\cT_Y)^{\perp\perp}$ preserves the tensor.
After all, we are just discussing the quotient of~$\cT$ by a $\otimes$-ideal: $\hTY\cong\cT/\cT_Y^\perp$.
\end{Rem}

\begin{Rem}
\label{Rem:tt-complete}%
By \Cref{Def:tt-completion}, an object~$t\in\cT$ is tt-complete (along~$Y$) when it belongs to~$(\cT_Y)^{\perp\perp}$, that is, when the unit map~$t\to \hat{t}_Y=[\eY,t]$ is an isomorphism in~$\cT$.
This is equivalent to the vanishing of~$[\fY,t]$ in~$\cT$, by applying~$[-,t]$ to~\eqref{eq:idemp-triangle}.
\end{Rem}

\begin{Rem}
\label{Rem:TY<->hTY}%
In view of \Cref{Rem:TY=hTY}, everything `tt' that we discuss in this article, like compact or dualizable objects, can be done interchangeably with the tt-category~$(\cT_Y)^{\perp\perp}$ or the tt-equivalent~$\cT_Y$.
Following tradition, we tend to \emph{state} our results about completion in terms of~$\hTY=(\cT_Y)^{\perp\perp}$. However we often \emph{prove} them in the more explicit~$\cT_Y$.
\end{Rem}

\begin{Rem}
The analogy between $\holim_n k_n^\vee\otimes t$ and~$\lim_n M/I^n M$ in \Cref{Rem:TY=hTY} may 	have caused more cognitive harm than good over the years, in particular in creating the illusion that $\hTY$ recovers the usual $I$-adic completion. Let us clear this up.
\end{Rem}

\begin{Rem}
\label{Rem:trivial-tt}%
By construction, $\cT_Y$ is compactly generated by~$\cT_Y^c$, which consists of dualizable objects as these are the images under the tt-functor~$\eY\otimes-\colon \cT\to \cT_Y$ of~$\cT^c_Y\subseteq\cT^d$. But its $\otimes$-unit $\unit_{\cT_Y}=\eY$ might not be compact.

The only option for~$\eY$ to be compact is for~$Y$ to be open. Indeed, if~$\eY$ is compact in~$\cT_Y$ then $\eY\in(\cT_Y)^c=(\cT^c)_Y\subseteq\cT^c$ and therefore the triangle~\eqref{eq:idemp-triangle} belongs to~$\cT^c$ and the right-idempotent~$\fY$ is compact as well; the relation~$\eY\otimes\fY=0$ gives~$\SpcT=\supp(\eY)\sqcup\supp(\fY)=Y\sqcup Y'$, where~$Y'=\supp(\fY)$ is closed; hence~$Y$ is open. Actually, more is true in that case. We must have $\unit\cong \eY\oplus\fY$ and the whole category $\cT=\cT_Y\times \cT_{Y'}=\hTY\times \cT_{Y'}$ decomposes, making the entire story rather trivial.
Compare with Stevenson~\cite{StevensonApp13}.

Of course, the same is true for the tt-equivalent category~$\hTY$ of {\Ycomplete} objects. See \Cref{Rem:TY<->hTY}. The tt-category~$\hTY$ is compactly generated and its compact objects are dualizable but its $\otimes$-unit is compact only if~$\cT$ splits as above.
\end{Rem}

\begin{Exa}
\label{Exa:trivial-Der}%
For instance, if~$\cT=\Der(R)$ is the derived category of a commutative ring~$R$, the units of~$\cT_Y$ and of~$\hTY$ are compact (inside those categories) if and only if~$R\simeq R_1\times R_2$ is the product of two rings, etc.
In particular, the derived category~$\smash{\Der(\hat\bbZ_p)}$ of the $p$-adics cannot be obtained as~$\Der(\bbZ)^\wedge_{(p)}$, nor as~$\Der(\bbZ_{(p)})^\wedge_{(p)}$, and not even as the $p$-complete objects $\Der(\hat\bbZ_p)^\wedge_{(p)}$ over~$\hat\bbZ_p$ itself! It really never works!
\end{Exa}

\begin{Rem}
In conclusion, calling $\hTY$ the {\Ycompletion} of~$\cT$ is a little bit of a misnomer. We revisit this topic in \Cref{Rem:final}.
\end{Rem}

We now want to assume that $\cT$ admits a tt-functor to another tt-category~$\cS$ that plays the role of $\cS=\Der(\hat{R})$ when~$\cT=\Der(R)$.
Let us recall the basics.
\begin{Rec}
\label{Rec:f^*}%
Consider a tensor-exact functor~$f^*\colon \cT\to \cS$ between two `big' tt-categories. Following~\cite{HoveyPalmieriStrickland97}, we say that $f^*$ is a \emph{geometric} functor if it preserves arbitrary coproducts. Since $f^*$ is a tensor functor, it preserves dualizable objects, hence compact objects. By Neeman~\cite{Neeman96}, $f^*$ admits a right adjoint~$f_*\colon \cS\to \cT$ which preserves coproducts and the adjunction $f^*\adj f_*$ satisfies a projection formula. See details in~\cite{BalmerDellAmbrogioSanders16}.

Denote by~$f:=\Spc(f^*)\colon \SpcS\to \SpcT$ the map induced by the tt-functor~$f^*\colon \cT^c\to \cS^c$ on small objects. Let $Y'=f\inv(Y)$ be the preimage of our Thomason subset~$Y$; this  $Y'$ is also a Thomason subset and $f^*\colon \cT\to \cS$ maps~$\cT_Y$ into~$\cS_{Y'}$. Moreover~$f^*(\eY)\cong \ee_{Y'}$ by \cite[Theorem~6.3]{BalmerFavi11}.
\end{Rec}

\begin{Thm}
\label{Thm:tt-compl}%
Let $f^*\colon \cT\to \cS$ be a geometric tt-functor and let~$Y'=f\inv(Y)$ as in \Cref{Rec:f^*}.
Suppose that $\cS$ satisfies the following two hypotheses:
\begin{enumerate}[\rm($\cS$1)]
\item
\label{it:S1}%
The unit of~$\cS$ is tt-complete along~$Y'$, that is, $\unit\cong\hat{\unit}_{Y'}$ in~$\cS$.
\smallbreak
\item
\label{it:S2}%
The right adjoint~$f_*\colon \cS\to \cT$ is conservative, or equivalently, $\cS^c$ is generated, as a thick triangulated subcategory, by the image~$f^*(\cT^c)$.
\end{enumerate}
Then the following three properties are equivalent:
\begin{enumerate}[\rm(i)]
\item
\label{it:tt-compl-1}%
There exists an isomorphism $\hat{\unit}_{Y}\simeq f_*(\unitS)$ of ring objects in~$\cT$.
\smallbreak
\item
\label{it:tt-compl-2}%
The functor~$f^*\colon \cT\to \cS$ is fully faithful on~$\cT_Y$.
\smallbreak
\item
\label{it:tt-compl-3}%
The functor $f^*\colon \cT\to \cS$ restricts to an equivalence $\cT_Y\isoto \cS_{Y'}$.
\end{enumerate}
Suppose furthermore that $Y$ is closed and that $k\in \cT^c$ is a compact object such that $\supp(k)=Y$. Then the above three properties are equivalent to:
\begin{enumerate}[\rm(i)]
\setcounter{enumi}{3}
\item
\label{it:tt-compl-4}%
The unit~$k\to f_*(f^*(k))$ is an isomorphism.
\end{enumerate}
\end{Thm}
\begin{Rem}
The proof will show that \eqref{it:tt-compl-1}$\then$\eqref{it:tt-compl-2} and \eqref{it:tt-compl-3}$\then$\eqref{it:tt-compl-2} hold without Hypotheses~($\cS$\ref{it:S1}) and~($\cS$\ref{it:S2}), and so does~\eqref{it:tt-compl-2}$\Leftrightarrow$\eqref{it:tt-compl-4} when $Y$ is closed. Hypothesis~($\cS$\ref{it:S1}) is used to prove \eqref{it:tt-compl-3}$\then$\eqref{it:tt-compl-1} and Hypothesis~($\cS$\ref{it:S2}) is used to prove \eqref{it:tt-compl-2}$\then$\eqref{it:tt-compl-3}.
\end{Rem}

\begin{proof}[Proof of \Cref{Thm:tt-compl}]
Suppose~\eqref{it:tt-compl-1}: We have a commutative triangle in~$\cT$
\[
\xymatrix{
& \unit \ar[ld]_-{\eta^{\otimes}} \ar[rd]^-{\eta_\unit}
\\
\hat{\unit}_Y \ar[rr]^-{\simeq}&&f_*(\unitS)
}
\]
where $\eta^{\otimes}\colon \unit\to [\eY,\eY]$ is the unit of the tensor-hom adjunction and $\eta_{\unit}$ is the unit of the $f^*\adj f_*$ adjunction at~$\unit$.
Since the former becomes an isomorphism under~$\eY\otimes-$ so does the latter: $\eY\otimes\eta_{\unit}$ is an isomorphism.
By the identification of the projection formula, the unit~$\eta_t\colon t\to f_*f^*(t)\cong t\otimes f_*(\unitS)$ is nothing but~$t\otimes \eta_{\unit}$.
We have shown that $\eta_t$ is an isomorphism for~$t=\eY$ and therefore for all~$t\in \eY\otimes \cT=\cT_Y$. This proves that $(f^*)\restr{\cT_Y}\colon \cT_Y\to \cS$ is fully faithful, as claimed in~\eqref{it:tt-compl-2}.

Suppose that $f^*$ is fully faithful on~$\cT_Y=\Loc(\cT^c_Y)$ as in~\eqref{it:tt-compl-2}. The essential image of~$(f^*)\restr{\cT_Y}\colon \cT_Y\to \cS$ is then generated as a localizing subcategory by~$f^*(\cT^c_Y)$. Thanks to Hypothesis~($\cS$\ref{it:S2}) the latter is a tensor-ideal in~$\cS^c$ and it coincides with~$\cS^c_{Y'}$ by \cite[Theorem~6.3]{BalmerFavi11} again. In short, $f^*(\cT_Y)=\Loc(\cS^c_{Y'})=\cS_{Y'}$ as claimed in~\eqref{it:tt-compl-3}.

Now suppose~\eqref{it:tt-compl-3}. We have a diagram
\[
\xymatrix@C=8em{
\cT_Y \ar[r]^-{G:=(f^*)\restr{\cT_Y}}_-{\cong} \ar@{ >->}@<-.5em>[d]_-{\incl}
& \cS_{Y'} \ar@{ >->}@<-.5em>[d]_-{\incl}
\\
\cT \ar[r]^-{f^*} \ar@{->>}@<-.5em>[u]_-{\eY\otimes-}
& \cS \ar@{->>}@<-.5em>[u]_-{\ee_{Y'}\otimes-}
}
\]
where the square with inclusions commutes by definition of~$G$. The other square also commutes, \ie we have an isomorphism of tt-functors from~$\cT$ to~$\cS_{Y'}$
\[
G\circ (\eY\otimes-) \cong (\ee_{Y'}\otimes-)\circ f^*
\]
since $f^*$ is a tensor functor and~$\ee_{Y'}\cong f^*(\eY)$.
Taking right adjoints in the latter isomorphism of tensor functors and using that $G\inv$ is the right adjoint of~$G$, we obtain an isomorphism of lax-monoidal functors $\cS_{Y'}\to \cT$:
\[
[\eY,-] \circ G\inv \cong f_*\circ [\ee_{Y'},-].
\]
Evaluating at the unit~$\ee_{Y'}=G(\eY)$ of~$\cS_{Y'}$ we obtain an isomorphism of ring objects
\[
[\eY,\eY] \cong f_*([\ee_{Y'},\ee_{Y'}])
\]
in~$\cT$. By Hypothesis~($\cS$\ref{it:S1}), the object~$[\ee_{Y'},\ee_{Y'}]=\hat{\unit}_{Y'}$ on the right-hand side agrees with~$\unitS$ and this gives us~\eqref{it:tt-compl-1}.

Finally, Condition~\eqref{it:tt-compl-4} is equivalent to the unit~$\eta_t\colon t\to f_*(f^*(t))$ being an isomorphism for all~$t$ in the localizing $\otimes$-ideal of~$\cT$ generated by~$k$, which is precisely~$\cT_Y$. Thus~\eqref{it:tt-compl-4} is equivalent to $f^*\colon \cT_Y\to \cS$ being fully faithful as in~\eqref{it:tt-compl-2}.
\end{proof}

\begin{Rem}
In \Cref{Thm:tt-compl}, we cannot drop Hypothesis~($\cS$\ref{it:S1}) for a trivial reason: Take~$f^*=\Id_{\cT}$ in a tt-category~$\cT$ whose unit is not tt-complete along~$Y$; in that case~\eqref{it:tt-compl-1} fails but the other properties are trivially true.

We cannot drop Hypothesis~($\cS$\ref{it:S2}) either.
To see this, consider the easy example of $Y=\SpcT$ itself, in which case~$\eY=\unit$ and tt-completeness is a void condition.
In that situation, a mere isomorphism $f_*(\unit_\cS)\simeq\unit_\cT$ does not force $f^*$ to be an equivalence, as for instance with $f^*\colon \Der(R)\to \Der(\mathbb{P}_{\hspace{-0.4ex}R}^1)$.
In that example, \eqref{it:tt-compl-1}, \eqref{it:tt-compl-2} and~\eqref{it:tt-compl-4} hold true but~\eqref{it:tt-compl-3} fails.
\end{Rem}

\begin{Rem}
\label{Rem:k-complete}%
Suppose that $Y$ is closed in~\Cref{Thm:tt-compl}.
Since~\eqref{it:tt-compl-1}--\eqref{it:tt-compl-3} are independent of~$k$ we see that if Property~\eqref{it:tt-compl-4} holds for one choice of $k$ in~$\cT^c_Y$ then it holds for all other choices. In other words, \eqref{it:tt-compl-4} only depends on~$Y$.
\end{Rem}

\begin{Cor}
\label{Cor:tt-compl}%
With notation as in \Cref{Thm:tt-compl}, let $\hat{f}^*\colon \hTY\to \SYp$ be the tt-functor $t\mapsto (f^*(t))^\wedge_{Y'}=[\ee_{Y'},f^*(t)]$ obtained by $Y'$-completion of~$f^*$.
Then the equivalent conditions of \Cref{Thm:tt-compl} are furthermore equivalent to
\begin{enumerate}[\rm(i')]
\setcounter{enumi}{1}
\item
\label{it:tt-compl-2'}%
The functor~$\hat{f}^*\colon \hTY\to \SYp$ is fully faithful.
\smallbreak
\item
\label{it:tt-compl-3'}%
The functor~$\hat{f}^*\colon \hTY\to \SYp$ is an equivalence.
\end{enumerate}
\end{Cor}
\begin{proof}
In view of \Cref{Rem:TY=hTY}, we can transport results from~$\cT_Y$ to~$\hTY$.
We claim that the following diagram of tt-functors commutes
\[
\xymatrix@C=4em{
\cT_Y =\eY\otimes\cT \ \ar[d]_-{f^*}
&
\ (\cT_Y)^{\perp\perp} = \hat \cT_Y \ar[l]^{\cong}_-{\eY\otimes-} \ \ar[d]_-{\hat f^*}
\\
\cS_{Y'} =\ee_{Y'}\otimes\cS \ar[r]^-{[\ee_{Y'},-]}_-{\cong}
&
\ (\cS_{Y'})^{\perp\perp} = \hat \cS_{Y'}
}
\]
where the horizontal tt-equivalences are the ones of~\eqref{eq:TY-vs-hTY}.
Commutativity uses $\ee_{Y'}\cong f^*(\eY)$ and direct computation $[\ee_{Y'},f^*(\eY\otimes t)]\cong[\ee_{Y'},f^*(\eY)\otimes f^*(t)]\cong[\ee_{Y'},\ee_{Y'}\otimes f^*(t)]\cong[\ee_{Y'},f^*(t)]$ which is the formula for~$\hat f^*=(-)^\wedge_{Y'}\circ f^*$.
So we can rephrase Properties~\eqref{it:tt-compl-2} and~\eqref{it:tt-compl-3} about $f^*\colon \cT_Y\to \cS_{Y'}$ in terms of~$\hat f^*$.
\end{proof}

At this stage, the reader probably wants some explicit example. Here it comes.


\section{Ring completion and Koszul complexes}
\label{sec:comm-alg}%


We now turn our attention to commutative algebra. Let $R$ be a commutative ring, which need not be noetherian.
We consider the `big' tt-category~$\cT=\Der(R)$.
We write $I\subset R$ for a finitely generated ideal and $Y=V(I)$ for the associated Thomason closed subset of~$\Spec(R)\cong\SpcT$.
We recall the idempotent triangle~\eqref{eq:idemp-triangle} $\eY\to \unit\to \fY\to \Sigma\eY$ of \Cref{Rec:idempotents}. (See more about this in \Cref{Rem:ef-in-DR}.)

\begin{Rem}
\label{Rem:I-fg}%
A tt-completion of~$\cT$ in the sense of \cref{sec:tt-completion} depends on the choice of a tt-ideal of compact objects, which is necessarily of the form~$\cT^c_Y$ for a unique \emph{Thomason} subset~$Y$ of~$\SpcT\cong\Spec(R)$; see~\cite{Thomason97}.
If, following tradition, one also wants $Y$ to be \emph{closed} then $Y$ will be given by $Y=V(I)$ for some \emph{finitely generated} ideal~$I$ as above.
In particular, non-finitely generated ideals~$I$ are never part of the tt-completion picture, even for non-noetherian rings~$R$.
\end{Rem}

\begin{Rec}
\label{Rec:Der-Y}%
It is well-known that $\cT_Y=\Loc(\Dperf(R)_Y)$ coincides with the subcategory $\DYR$ of complexes of~$R$-modules supported on~$Y$, that is, whose restriction to the open complement~$U:=\Spec(R)\sminus Y$ is acyclic. See~\cite{Neeman92b}.
The tt-complete complexes of~\Cref{Def:tt-completion}, \ie those complexes belonging to
\[
\hTY=(\cT_Y)^{\perp\perp}=\SET{t\in\Der(R)}{t\cong[\eY,t]}
\]
are often called \emph{derived $I$-complete} or \emph{derived {\Ycomplete}}.
The category $\hTY=\Der(R)^\wedge_Y$ is often denoted $\Der(R)_{\mathrm{comp}}$ or $\Der(R;I)_{\mathrm{comp}}$ in the literature.
Finally, recall the \mbox{tt-equivalence} $\DYR\cong\Der(R)^\wedge_Y$ of \Cref{Rem:TY=hTY}.
\end{Rec}

\begin{Rem}
There is a vast literature comparing classical $I$-adic completion $\hat{M}_I=\lim_n M/I^nM$, its derived functors, and tt-completion~$[\eY,-]$, for $R$-modules and for complexes.
It is rich and technical and we cannot do it justice here. The interested reader will find a more detailed account in the Stack Project~\cite[\href{https://stacks.math.columbia.edu/tag/091N}{Section~15.91}]{stacks-project}, referring to earlier work such as~\cite{GreenleesMay92,DwyerGreenlees02,PortaShaulYekutieli14}, to which one could add \cite{Schenzel03,Positselski23} for the derived functors of completion.
We recall some facts.
\end{Rem}

\begin{Ter}
\label{Ter:complete}%
An $R$-module~$M$ is called \emph{classically $I$-complete} if $M\cong\hat{M}_I$.
It is called \emph{$I$-separated} if $\cap_{n}I^nM=0$.
\end{Ter}

\begin{Prop}[{\cite[Proposition~15.91.5]{stacks-project}}]
\label{Prop:classical=>derived}%
An $R$-module $M$ is classically $I$-complete if and only if it is $I$-separated and derived $I$-complete (\Cref{Rec:Der-Y}).
In particular, if~$R\cong\lim_n R/I^n$ is classically $I$-complete then $\unit\cong[\eY,\unit]=\hat{\unit}_Y$, that is, the unit of~$\Der(R)$ is tt-complete along~$Y$.
\qed
\end{Prop}

Let us make the principal case explicit for the benefit of the reader.
\begin{Exa}
\label{Exa:mono}
Let $s\in R$ and~$Y=V(s)$. In that case, localization away from~$V(s)$ amounts to passing from~$R$ to~$R[1/s]$ and the right idempotent~$\ff_s=\fY$ is simply~$\unit[1/s]=\hocolim(\unit\xto{s}\unit\xto{s}\cdots)$, that is, the complex
\begin{equation}
\label{eq:f_s}%
\ff_s=\big(\cdots \to 0 \to \coprod_{n\in\bbN}R \xto{\id-\tau} \coprod_{n\in\bbN}R \to 0 \to \cdots \big)
\end{equation}
concentrated in homological degrees one and zero, where $\tau$ maps the $n$-th copy of~$R$ to the $(n+1)$-st copy of~$R$, via multiplication by~$s$.
(To be clear~$\bbN=\{0,1,2,\ldots\}$.)
In other words, this complex is a projective resolution of~$R[1/s]$. The canonical map $\unit\to \ff_s$ is the inclusion of~$R$ in the 0-th copy of~$R$ in degree zero.
Consequently $\ee_{V(s)}$ can also be made explicit: It is the homotopy fiber of this map $\unit\to \ff_s$.

For any module~$M$, for instance~$M=R$, we know that $M$ is derived $s$-complete if and only if~$[\ff_s,M]=0$.
Since we described~$\ff_s$ as a bounded complex of projectives in~\eqref{eq:f_s}, we can compute~$[\ff_s,M]$ by applying $\Hom_R(-,M)$ degreewise. We get
\begin{equation}
\label{eq:[fsM]}%
[\ff_s,M]=\big(\cdots \to 0 \to \prod_{n\in\bbN}M \xto{(\id-\tau)^*} \prod_{n\in\bbN}M \to 0 \to \cdots \big)
\end{equation}
concentrated in homological degrees zero and~$-1$; the map~$(\id-\tau)^*\colon M^{\bbN}\to M^{\bbN}$ is easily checked to send $(x_n)_{n\in\bbN}$ to~$(x_n-sx_{n+1})_{n\in\bbN}$.

In conclusion, $M$ is derived $s$-complete if and only if the map~$(\id-\tau)^*$ in~\eqref{eq:[fsM]} is an isomorphism.
It is easy to verify that this holds if $M$ is classically $I$-complete.
In fact, this holds if $M$ is classically $I$-complete for any ideal~$I$ that contains~$s$.
(For surjectivity of~$(1-\tau)^*$, show that every $(y_n)_n$ is the image of $(x_n)_n$ defined by~$x_n=\sum_{\ell=0}^\infty s^\ell y_{\ell+n}$. For injectivity, use~$\cap_ns^nM\subseteq\cap_nI^nM=0$.)
\end{Exa}

\begin{Rem}
\label{Rem:ef-in-DR}%
We gave explicit formulas for the idempotents~$\ee_{V(s)}$ and~$\ff_{V(s)}$ in the monogenic case above.
For a general ideal $I=\ideal{s_1,\ldots,s_r}$ we have $\ee_{V(I)}=\ee_{V(s_1)}\otimes\cdots\otimes\ee_{V(s_r)}$ and one can recover $\ff_{V(I)}$ as the mapping cone of~$\ee_{V(I)}\to \unit$. In summary, the idempotents in the triangle~\eqref{eq:idemp-triangle} can be made very concrete in~$\cT=\Der(R)$.
See also \Cref{Prop:eY} for another description.
\end{Rem}

\begin{Rem}
\label{Rem:ring-compl}%
We have seen in \Cref{Prop:classical=>derived} that classically $I$-complete implies derived $I$-complete. (We also proved it in the monogenic case at the end of \Cref{Exa:mono} and one can jazz that up into a complete proof using an induction argument on the number of generators, thanks to the Mayer--Vietoris properties of idempotents~\cite[Theorem~5.18]{BalmerFavi11}. We leave this to the interested reader.)
As can be expected, the converse is wrong: There are modules~$M$ that are derived $I$-complete but not classically $I$-complete. See Yekutieli~\cite[Example~3.20]{Yekutieli11}.
\end{Rem}

\begin{Exa}
\label{Exa:no<=modules}%
Kedlaya~\cite[\href{https://kskedlaya.org/prismatic/sec_derived-complete.html}{Example~6.1.4}]{prismatic-online} gives a simple example over the ring~$\OO=\hat\bbZ_p$ of $p$-adics.
Take $L=\SET{(x_i)_i\in \OO^\bbN}{\lim_i x_i=0}$ to be the $\OO$-module of zero-converging sequences and let $M=\mathrm{Coker}(\ell\colon L\to L)$ be the cokernel of the monomorphism~$\ell((x_i)_i)\coloneqq(p^i x_i)_i$. Then $M$ is derived $p$-complete, as $L$ is (\eg\ by the criterion of \Cref{Exa:mono}). However, $M$ is not $p$-separated: For instance, the sequence $(p^i)_i$ is a non-zero element of~$\cap_n p^n M$.
\end{Exa}

There are also rings~$R$ that are derived complete but not classically complete.
\begin{Exa}
\label{Exa:no<=rings}%
Let $\OO$ be an $s$-complete ring for some $s\in \OO$ and let $M$ be a derived $s$-complete module that is not $s$-separated (for instance $s=p$ in $\OO=\hat\bbZ_p$ and $M$ as in \Cref{Exa:no<=modules}).
Define the $\OO$-algebra $R=\OO\oplus M$ in which $M$ squares to zero.
Then~$R$ is derived $s$-complete but not classically $s$-complete. The latter is easy, since $\cap_n s^n\cdot R$ contains $\cap_n s^n M\neq0$ in the second factor. So $R$ is not $s$-separated. To see that
the derived completeness of~$\OO$ and of~$M$ over~$\OO$ implies that the square-zero extension~$R$ is derived $s$-complete as an $R$-module,
we can use the discussion of~\Cref{Exa:mono}.
It suffices to show that the map $(\id-\tau)^*$ of~\eqref{eq:[fsM]} is an isomorphism $R^\bbN\to R^\bbN$. This map separates into two factors, one involving~$\OO^{\bbN}$ and one involving~$M^\bbN$, which are both isomorphisms by the derived $s$-completeness of~$\OO$ and of~$M$ over~$\OO$.
\end{Exa}

\begin{Rem}
In summary, we have seen that if the ring homomorphism~$R\to \hat{R}_{{}}$ is an isomorphism then so is the map~$\unit\to \hat{\unit}_Y$ in~$\Der(R)$. We have seen that the converse fails (\Cref{Exa:no<=rings}). Consequently, there is no ring isomorphism $\hat{\unit}_Y\simeq\hat{R}_{{}}$ in general.
Hammering the point: For the ring of \Cref{Exa:no<=rings}, the derived completion~$\hat{\unit}_Y$ is not recovering the classical $I$-adic completion~$\hat{R}_{}$.
We want a simple formulation for when these two notions of completion agree.
\end{Rem}

\begin{Def}
\label{Def:I-kos}%
We say that a sequence $\underline{s}=(s_1,\ldots,s_r)$ in~$R$ is \emph{Koszul-complete} if the canonical map~$R\to \hat{R}_{}=\hat{R}_{I}$ for $I=\langle s_1,\ldots,s_r\rangle$ induces a quasi-isomorphism
\[
\kos(\underline{s})\to \kos(\underline{s})\otimes_R\hat{R}_{}
\]
where~$\kos(\underline{s})=\otimes_{i=1}^r\cone(s_i\colon \unit\to \unit)$ is the usual Koszul complex.
Note that the above tensor could equivalently be a derived tensor since the Koszul complex is perfect.
Once we prove in \Cref{Rem:Kos-complete} that this property only depends on~$Y=V(I)$, not on the choice of~$\underline{s}$, we could also simply say that $Y$ is \emph{Koszul-complete}.
\end{Def}

\begin{Rem}
The induced map in homology~$\rmH_i(\kos_R(\underline{s}))\to \rmH_i(\kos_R(\underline{s})\otimes_R\hat{R}_{})\cong\rmH_i(\kos_{\hat{R}_{}}(\underline{s}))$ is an isomorphism for~$i=0$, since both sides are $R/I$ in that case. Thus~$\underline{s}$ being Koszul-complete means that these maps are isomorphisms for~$i=1,\ldots,r$.
\end{Rem}

\begin{Exa}
\label{Exa:no}%
This property does not hold in general, already for~$r=1$.
For instance, it would fail for the ring of \Cref{Exa:no<=rings}.
One can give a simpler example in the same vein.
Let $p$ be a prime and consider the $\bbZ_{(p)}$-algebra $R=\bbZ_{(p)}\oplus (\bbQ/\bbZ_{(p)})$ with the second term squaring to zero.
We use $s_1=p$.
Observe that $\rmH_1(\kos(p))=\Ker(R\xto{p}R)$ is non-zero; it is~$\bbZ/p$ embedded in~$\bbQ/\bbZ_{(p)}$ via~$1/p$. Since $R/p^n\cong\bbZ/p^n$ for all~$n\geq 1$, the completion $\smash{\hat{R}_{(p)}}\cong\smash{\hat{\bbZ}_{(p)}}$ is a domain. So $\kos(p)\otimes_R\smash{\hat{R}_{(p)}}$ has trivial~$\rmH_1$ and the map~$\kos(s)\to \kos(s)\otimes_R\hat{R}_{(p)}$ is not injective on~$\rmH_1$.
\end{Exa}

This notion of Koszul-completeness is only relevant in the non-noetherian world:
\begin{Prop}\label{prop:noeth-koszul}
Suppose that the ring $R$ is noetherian. Then every sequence $\underline{s}=(s_1,\ldots,s_r)$ is Koszul-complete in the sense of \Cref{Def:I-kos}.
\end{Prop}
\begin{proof}
Although probably well-known, we give a proof for completeness.
Write $I=\langle s_1,\ldots,s_r\rangle$ and recall that because $R$ is noetherian, we have $M\otimes_R \hat{R}_{}\cong \hat{M}_I$ for every finitely generated $R$-module~$M$.
We can apply this to one of our homology groups $M=\rmH_i(\kos(\underline{s}))$ for $1\le i\le r$, which is finitely generated because $R$ is noetherian.
On the other hand, the homology of the Koszul complex is killed by~$s_1,\ldots,s_r$, since the maps~$s_i\cdot-$ are homotopic to zero on $\cone(s_i)$ hence on the Koszul complex itself. So our $R$-module~$M=\rmH_i(\kos(\underline{s}))$ satisfies $IM=0$ and thus~$\hat{M}_I\cong M$.
Therefore $\rmH_i(\kos(\underline{s}))=M\cong \hat{M}_I\cong M\otimes_R\hat{R}_{}=\rmH_i(\kos(\underline{s}))\otimes_R\hat{R}_{}\cong
\rmH_i(\kos(\underline{s})\otimes_R\hat{R}_{})$ where the last isomorphism holds because $\hat{R}$ is $R$-flat in the noetherian case.
\end{proof}

\begin{Exa}
If $\underline{s}$ is regular and $\hat{R}_{}$ is flat over~$R$ then $\kos(s)\to R/I$ is a quasi-isomorphism, hence $\kos(s)\otimes\hat{R}_{}\to R/I\otimes \hat{R}_{}$ is an isomorphism by flatness. But $R/I\otimes_R\hat{R}_{}\cong\hat{R}_{}/I\cong R/I$ so $s$ is Koszul-complete, even without $R$ being noetherian.
\end{Exa}

\begin{Rem}\label{Rem:fg-comp-is-comp}
A general fact about classical completion with respect to a finitely generated ideal $I$ is that the completion is itself always complete; that is, the completion $\hat{R}_{}$ is always classically $I'$-complete for the extended (finitely generated) ideal $I'=I\hat{R}_{}$; see \cite[\href{https://stacks.math.columbia.edu/tag/05GG}{Lemma~05GG}]{stacks-project}. This does not require $R$ to be noetherian.
\end{Rem}

\begin{Thm}
\label{Thm:D(R)-completion}%
Let $R$ be a commutative ring, $I\subseteq R$ a finitely generated ideal and~$Y=V(I)$.
Let $f\colon\Spec(\hat{R}_{})\to \Spec(R)$ and let $Y'=f\inv(Y)=V(I')$ be the preimage of~$Y$. (See \Cref{Rem:fg-comp-is-comp}.)
Then the following conditions are equivalent:
\begin{enumerate}[\rm(i)]
\item
\label{it:D(R)-compl-1}%
There exists an isomorphism of ring objects~$[\eY,\unit]\simeq\hat{R}_{}$ in~$\Der(R)$.
\smallbreak
\item
\label{it:D(R)-compl-2}%
The functor~$f^*\colon \Der(R)\to \Der(\hat{R}_{})$ restricts to an equivalence $\Der_Y(R)\isoto \Der_{Y'}(\hat{R}_{})$.
\smallbreak
\item
\label{it:D(R)-compl-2'}%
The completed extension-of-scalars~$\hat f^*\colon \Der(R)^{\wedge}_{Y}\to \Der(\hat{R}_{})^{\wedge}_{Y'}$ is an equivalence.
\smallbreak
\item
\label{it:D(R)-compl-3}%
Every sequence $\underline{s}$ that generates~$I$ is Koszul-complete (\Cref{Def:I-kos}).
\smallbreak
\item
\label{it:D(R)-compl-4}%
There exists a Koszul-complete sequence $\underline{s}$ that generates~$I$.
\end{enumerate}
\end{Thm}
\begin{proof}
It suffices to apply \Cref{Thm:tt-compl} and \Cref{Cor:tt-compl} to~$\cT=\Der(R)$ and~$\cS=\Der(\hat{R}_{})$ and $f^*\colon\cT\to \cS$ the extension-of-scalars
and to the object~$k=\kos(\underline{s})$, which generates~$\cT^c_Y$ by \cite{Thomason97}.
The two hypotheses~($\cS$\ref{it:S1}) and~($\cS$\ref{it:S2}) hold. The latter is obvious since $\hat{R}_{}=f^*(R)$ generates~$\cS^c=\Dperf(\hat{R}_{})$. The former holds by \Cref{Prop:classical=>derived} since~$\hat{R}_{}$ is $I'$-complete, where $I'=I\hat{R}_{}$ is finitely generated (\cref{Rem:fg-comp-is-comp}).
\end{proof}

\begin{Rem}
\label{Rem:Kos-complete}%
Following up on~\Cref{Rem:k-complete}, we see from \Cref{Thm:D(R)-completion} that once we know that an ideal $I$ is generated by a Koszul-complete sequence then every other sequence of generators is automatically Koszul-complete. In fact, the same holds for any sequence defining the same closed subset~$Y$.
\end{Rem}

\begin{Cor}
If $R$ is noetherian then we have tt-equivalences on {\Ytorsion} $\Der_Y(R)\isoto \Der_{Y'}(\hat{R})$ and on~{\Ycomplete} $\Der(R)^\wedge_Y\simeq\Der(\hat{R}_{})^\wedge_{Y'}$ subcategories.
\end{Cor}
\begin{proof}
This is an immediate consequence of \cref{prop:noeth-koszul} and \cref{Thm:D(R)-completion}.
\end{proof}


\section{Dualizable objects in the complete case}
\label{sec:dualizables-in-TY}%


We keep our setup: $R$ is a commutative ring and $Y\coloneqq V(I)\subset\Spec(R)$ for a finitely generated ideal~$I$.
The goal of this technical section is to characterize the dualizable objects in the derived category~$\Der_Y(R)$ of complexes supported on~$Y$, when $R\cong\hat{R}$ is classically $I$-adically complete.
It will be achieved in \cref{Thm:dualizable-complete}.
Most of the discussion holds for any commutative ring~$R$.

\begin{Rem}
\label{Rem:H0}%
We use homological indexing for complexes and view $R$-modules as complexes concentrated in degree zero.
We denote the (left-derived) tensor product in~$\cT=\Der(R)$ by~$\Lotimes$ and write $\otimes_R$ for the ordinary tensor product of $R$-modules.
Recall that if $t$ and~$u$ are right-bounded complexes, say, $t$ belongs to~$\Der_{\ge a}(R):=\SET{t\in\Der(R)}{\rmH_i(t)=0\textrm{ for all }i<a}$ and~$u\in\Der_{\ge b}(R)$, then $t\otimes u\in\Der_{\ge a+b}(R)$ is also right-bounded and the \emph{rightmost} homology modules satisfy
\[
\rmH_{a+b}(t\otimes u)\cong \rmH_{a}(t)\otimes_R \rmH_{b}(u)
\]
by the usual canonical truncations and right-exactness of~$\otimes_R$.
\end{Rem}

\begin{Rec}
\label{Rem:Mittag-Leffler}%
Let $\cdots \to t_{n+1}\xto{f_{n+1}} t_n \xto{f_{n}} t_{n-1}\to \cdots$ be a sequence of morphisms in a triangulated category~$\cT$, indexed by~$n\in\bbN$.
If $u\in\cT$ is an object such that every map $\Homcat{T}(\Sigma u,f_{n})\colon \Homcat{T}(\Sigma u,t_{n})\to \Homcat{T}(\Sigma u,t_{n-1})$ is surjective,
then there is a canonical isomorphism
\[
\Homcat{T}(u,\holim_n t_n)\cong \lim_n\Homcat{T}(u,t_n).
\]
This a standard Mittag-Leffler argument, combined with applying the homological functor $\Homcat{T}(u,-)$ to the exact triangle
\begin{equation}
\label{eq:holim}%
\holim_n t_n \to \prod_n t_n\xto{1-\tau} \prod_n t_n \to \Sigma \holim_n t_n
\end{equation}
that defines the homotopy limit. (As usual, $\tau$ is the unique map to the product such that~$\proj_n\circ\tau=f_{n+1}\circ \proj_{n+1}$ for all~$n\in\bbN$.)
\end{Rec}

We also remind the reader of the standard trick of lifting idempotents:

\begin{Prop}
\label{Prop:lift-proj}%
We can lift finitely generated projective modules, up to isomorphism, along the ring quotient~$\hat{R}_{}\onto R/I$.
\end{Prop}
\begin{proof}
It suffices to show that every idempotent matrix in~$\Mat_\ell(R/I)$ lifts to an idempotent matrix in~$\Mat_{\ell}(\hat{R}_{})$, for any~$\ell\ge1$. For each $n\ge 1$ the ring homomorphism
\[
\Mat_{\ell}(R/I^{n+1})\onto \Mat_{\ell}(R/I^n)
\]
has nilpotent kernel, thus idempotents lift.
Hence we can construct the desired idempotent in~$\lim_n\Mat_{\ell}(R/I^n)\cong\Mat_{\ell}(\lim_n R/I^n)=\Mat_{\ell}(\hat{R}_{})$.
\end{proof}


After these recollections, we gather in one place all the notation we shall use.

\begin{Not}
\label{Not:general-completion}%
Take $s_1,\ldots,s_r\in R$ that generate~$I$. Let $n\ge1$.
\begin{enumerate}[\rm(1)]
\item
Write $\In$ for~$\langle s_1^n,\ldots,s_r^n\rangle$. Of course $I^{(1)}=I$ and $\In\subseteq\Inm$.
\smallbreak
\item
\label{it:not-kos}%
We denote the Koszul complex for~$(s_1^n,\ldots,s_r^n)$ by
\[
\quad \kn:=\otimes_{i=1}^r\cone(s_i^n)=\otimes_{i=1}^r(\cdots\to 0\to R\xto{s_i^n}R\to 0\cdots)
\]
with non-zero entries in nonnegative homological degrees. (Each factor lives in degrees~1 and~0 and the product ranges between degree~$r$ and~$0$.)
We have maps~$q_n\colon \kn\to \knm$ given on each tensor factor $\cone(s_i^n)\to \cone(s_i^{n-1})$ by
\begin{equation*}
\quad\vcenter{\xymatrix@R=1em{
\cdots 0\ar[r]
& R \ar[r]^-{s_i^{n}} \ar[d]_-{s_i}
& R \ar[r] \ar@{=}[d]
& 0 \cdots
\\
\cdots 0\ar[r]
& R \ar[r]^-{s_i^{n-1}}
& R \ar[r]
& 0 \cdots
}}
\end{equation*}
\smallbreak
\item
There is a canonical map $p_n\colon \kn\to R/\In$ given by the projection $R\onto R/\In$ in degree zero
and the obvious square with the above~$q_n$ commutes:
\begin{equation}
\label{eq:pq-comm}%
\quad\vcenter{\xymatrix@C=2em@R=2em{
\kn \ar[r]^-{p_n} \ar[d]_-{q_n}
& R/\In  \ar[d]^-{\mathrm{can}}
\\
\knm \ar[r]^-{p_{n-1}}
& R/\Inm.
}}
\end{equation}
We define the complex $\ln\in\Der(R)$ as the homotopy fiber of~$p_n$, \ie by the exact triangle in~$\Der(R)$:
\begin{equation}
\label{eq:ln-triangle}%
\quad\ln \to \kn \xto{p_n} R/\In \to \Sigma \ln.
\end{equation}
Since $\rmH_0(\kn)=R/\In$, the long exact sequence in homology tells us that
\[
\ln\in\Der_{\ge 1}(R).
\]
\smallbreak
\item
We let~$\ion\colon \Spec(R/\In)\hook \Spec(R)$ denote the closed immersion associated to the ring surjection~$R\onto R/\In$ and let~$\ion^*\colon \Der(R)\to \Der(R/\In)$ denote the corresponding extension-of-scalars.
Since its right adjoint, restriction-of-scalars~$(\ion)_*$, preserves homology, we get via the projection formula
\[
\rmH_i(\ion^*(t))\cong\rmH_i((\ion)_*\ion^*(t))\cong\rmH_i((R/\In)\Lotimes t)
\]
for every $i\in \bbZ$ and~$t\in\Der(R)$.
By~\Cref{Rem:H0}, the tt-functor $\ion^*$ preserves right-boundedness: $\ion^*(\Der_{\ge 0}(R))\subseteq\Der_{\ge 0}(R/\In)$. Moreover, when~$t\in\Der_{\ge 0}(R)$ is concentrated in non-negative degrees we obtain
\begin{equation}
\label{eq:H0}%
\rmH_0(\ion^*(t))\cong
\rmH_0((R/\In)\Lotimes t)\cong
(R/\In)\otimes_R\rmH_0(t)\cong \rmH_0(t)\big/\In\cdot\rmH_0(t).
\end{equation}
\end{enumerate}
\end{Not}

We now give a series of preparatory lemmas.
Note that the Koszul complexes~$\kn$ of \Cref{Not:general-completion}\,\eqref{it:not-kos} are dualizable
and the maps~$q_n\colon \kn\to\knm$ described in~\eqref{it:not-kos} correspond to
maps~$q_n^\vee\colon (\knm)^\vee \to (\kn)^\vee$.
\begin{Prop}
\label{Prop:eY}%
The homotopy colimit of the sequence
\[
(k^{(1)})^\vee\xto{q_2^\vee} \cdots \xto{q_n^\vee} (\kn)^\vee\xto{q_{n+1}^\vee} (\knp)^\vee\to\cdots
\]
in~$\Der(R)$ is isomorphic to the left idempotent~$\eY$.
\end{Prop}
\begin{proof}
This is well-known.
Since $\eY=\ee_{V(s_1)}\otimes\cdots\otimes\ee_{V(s_r)}$, it is possible to reduce to the case $r=1$.
For $s\in R$, the right-idempotent~$\ff_{V(s)}$ is simply~$R[1/s]$ in degree zero, hence $\ee_{V(s)}$ is~$(0\to R\to R[1/s]\to 0)$ in homological degrees~0 and~$-1$, which is easily seen to be the homotopy colimit of
\[
\xymatrix@R=1em{
\cone(s)^\vee= \ar@<-.5em>[d]
&& \cdots 0 \ar[r]
& R \ar[r]^-{s} \ar@{=}[d]
& R \ar[r] \ar[d]^-{s}
& 0 \cdots
\\
\cone(s^2)^\vee= \ar@<-.5em>[d]
&& \cdots 0 \ar[r]
& R \ar[r]^-{s^2} \ar@{=}[d]
& R \ar[r] \ar[d]^-{s}
& 0 \cdots
\\
\cone(s^3)^\vee= \ar@<-.5em>[d]
&& \cdots 0 \ar[r]
& R \ar[r]^-{s^3} \ar@{=}[d]
& R \ar[r] \ar[d]^-{s}
& 0 \cdots
\\
{\vdots}\quad
&&&
{\vdots}
&{\vdots}
}
\]
\vskip-\baselineskip
\end{proof}
\begin{Cor}
\label{Cor:[et]}%
Let~$t\in\Der(R)$ be an object.
Then the homotopy limit of the sequence
\[
\cdots \to \knp\Lotimes t \xto{q_{n+1}\Lotimes 1} \kn\Lotimes t\xto{q_n\Lotimes 1} \cdots \xto{q_2\Lotimes1} k^{(1)}\otimes t
\]
is isomorphic to~$[\eY,t]$.
\end{Cor}
\begin{proof}
Apply~$[-,t]$ to \Cref{Prop:eY} and use $[(\kn)^\vee,t]\cong \kn\Lotimes t$.
\end{proof}

We now turn to the objects~$\ln$ defined by the exact triangles~\eqref{eq:ln-triangle}.
\begin{Lem}
\label{Lem:An}%
For every $n\ge 1$ we have $\In\cdot \rmH_i(\ln)=0$ for all~$i\ge1$.
\end{Lem}
\begin{proof}
As the Koszul complex $\kn$ and the module~$R/\In$ have the same homology in degree~0 and as $R/I$ is concentrated in degree~0, we have~$\rmH_i(\ln)=\rmH_i(\kn)$ for all~$i\ge 1$. But the map~$s_i^n$ is zero on~$\cone(s_i^n)$ and \textsl{a fortiori} on~$\otimes_{i}\cone(s_i^n)=\kn$. The induced map in homology is then zero as well.
\end{proof}

\begin{Lem}
\label{Lem:1}%
If $t\in \Der_{\ge0}(R)$ satisfies $\rmH_0((R/\In)\Lotimes t)=0$
then $\rmH_0(\kn\Lotimes t )=0$ as well and $\rmH_1(p_n\Lotimes 1 )\colon \rmH_1(\kn\Lotimes t )\to \rmH_1((R/\In)\Lotimes t)$ is an isomorphism.
\end{Lem}
\begin{proof}
Since $t\in\Der_{\ge 0}(R)$, we can use~\eqref{eq:H0} and the assumption $\rmH_0((R/\In)\Lotimes t)=0$ to deduce that $\rmH_0(t)=\In\cdot\rmH_0(t)$.
It follows that $\rmH_1(\ln)\otimes_R \rmH_0(t)=0$ since $\In\cdot \rmH_1(\ln)=0$ by \Cref{Lem:An}.
Since $\ln\in\Der_{\ge1}(R)$ and~$t\in\Der_{\ge 0}(R)$, we know that~$\ln\Lotimes t\in\Der_{\ge1}(R)$ and, by \Cref{Rem:H0} again, the rightmost homology is $\rmH_1(\ln\Lotimes t)\cong \rmH_1(\ln)\otimes_R \rmH_0(t)$.
As the latter vanishes, we have in fact
\[
\ln\Lotimes t\in\Der_{\ge2}(R).
\]
Plugging this into the homology long exact sequence associated to the exact triangle~\eqref{eq:ln-triangle}~$\Lotimes \,t$, we get isomorphisms $\rmH_i(\kn\Lotimes t )\isoto \rmH_i((R/\In)\Lotimes t)$ for $i=0,1$.
\end{proof}

\begin{Lem}
\label{Lem:2}%
Let $N\subset R$ be a nilideal and $a\in\bbZ$. Let $c\in\Dperf(R)$ be a perfect complex and~$d\in\Dperf(R/N)$ its image modulo~$N$.
Suppose that $d\in\Der_{\ge a}(R/N)$ has homology only in degree above~$a$. Then $c\in\Der_{\ge a}(R)$ has the same right-boundedness.
\end{Lem}
\begin{proof}
	We can assume $c\neq 0$. Let $\rmH_i(c)$ be the rightmost non-zero homology group of~$c$.
We need to show that $i\ge a$.
Since $c$ is perfect, its rightmost homology group~$\rmH_i(c)$ is a finitely generated $R$-module and, since $N$ is a nilideal, Nakayama's Lemma~guarantees that $(R/N)\otimes_R\rmH_i(c)$ remains non-zero.
By \Cref{Rem:H0}, this non-zero module is $\rmH_i((R/N)\Lotimes c)\cong\rmH_i(d)$. Hence $i\ge a$.
\end{proof}

We now turn to a key lemma, where derived completion appears.
\begin{Lem}
\label{Lem:3}%
Let $t\in \Der_{\ge0}(R)$ be such that $\ion^*(t)\in\Der(R/\In)$ is perfect for every~$n\ge1$.
If $\rmH_0((R/I)\Lotimes t)=0$ then $\rmH_0([\eY,t])=0$.
\end{Lem}
\begin{proof}
We first claim that $\rmH_0(\ion^*(t))=0$ for all~$n\ge 1$.
The $n=1$ case is the hypothesis~$\rmH_0(\iota_{1}^*(t))\cong \rmH_0((R/I)\Lotimes t)=0$.
Let $n\ge 2$ be such that $\rmH_0(\ionm^*(t))=0$.
Write $\bar{R}=R/\In$ and~$N=\Inm/\In\subseteq\bar{R}$. The homomorphism~$\bar{R}\onto \bar{R}/N=R/\Inm$ has nilpotent kernel~$N$ and sends~$\ion^*(t)$ to~$\ionm^*(t)$:
\[
\xymatrix@R=1em{
\ion^*(t) \ar@{|->}[d] & \in & \Dperf(R/\In) =\Dperf(\bar{R}) \quad \ar@<-3em>[d]
\\
\ionm^*(t) & \in & \Dperf(R/\Inm)=\Dperf(\bar{R}/N).
}
\]
Both of these objects are perfect complexes by assumption. We can therefore use \Cref{Lem:2} to conclude that $\rmH_0(\ion^*(t))=0$, which proves the claim by induction.

Our second claim is that the map~$\rmH_1((R/\In)\otimes t)\to \rmH_1((R/\Inm)\otimes t)$ is surjective for every $n\ge 2$. Keeping the notation~$\bar{R}=R/\In$ and~$N=\Inm/\In$ above, the perfect complex $\ion^*(t)$ belongs to~$\Der_{\ge 1}(R/\In)$ by the already proved first claim, \ie the perfect complex~$\ion^*(t)$ has rightmost non-zero homology in degree~one (or higher).
Thus we can invoke~\Cref{Rem:H0} again, for the ring~$\bar{R}$, to get that
\[
\rmH_1(\ionm^*(t))\cong
(\bar{R}/N)\otimes_{\bar{R}}\rmH_1(\ion^*(t)).
\]
Hence the map~$\rmH_1((R/\In)\otimes t)\cong\rmH_1(\ion^*(t))\onto \rmH_1(\ionm^*(t))\cong\rmH_1((R/\Inm)\otimes t)$ is indeed surjective.

Combining with~\eqref{eq:pq-comm} and \Cref{Lem:1} our two claims in degree zero and one also hold if we replace $(R/\In)\Lotimes t$ by~$\kn\Lotimes t$, namely we have
\[
\rmH_0(\kn\Lotimes t)\cong 0
\]
for all~$n\ge 1$ and the maps induced by~$q_n\colon \kn\to \knm$
\[
\rmH_1(\kn\otimes t)\to \rmH_1(\knm\otimes t)
\]
are surjective for all~$n\ge2$.
The latter implies that $\rmH_0(\holim_n(\kn\otimes t))\cong \lim_n \rmH_0(\kn\otimes t)=0$; see~\Cref{Rem:Mittag-Leffler} for $u=R$ in~$\cT=\Der(R)$.
But \Cref{Cor:[et]} tells us that this homotopy limit,~$\holim_n(\kn\otimes t)$, is precisely the object~$[\eY,t]$.
\end{proof}

\begin{Lem}
\label{Lem:[ed]}%
Let $b\ge a$ be integers. Let $d\in\Der(R)$ be such that $\kos(s_1,\ldots,s_r)\otimes d$ has homology concentrated in degrees between~$b$ and~$a$. Then:
\begin{enumerate}[\rm(a)]
\item
\label{it:[ed]-1}%
For every $n\ge 1$, the complex~$\kn\otimes d$ has homology concentrated in degrees between~$b$ and~$a$.
\smallbreak
\item
\label{it:[ed]-2}%
The complex~$[\eY,d]$ has homology concentrated in degrees between~$b$ and~$a-1$.
In particular, $[\eY,d]\in\Der_{\ge a-1}(R)$ is right-bounded.
\end{enumerate}
\end{Lem}
\begin{proof}
Part~\eqref{it:[ed]-1} follows from the more general fact that $\kos(s_1^{\ell_1},\ldots,s_r^{\ell_r})\otimes d$ has homology concentrated between degrees~$b$ and~$a$, for every $\ell_1,\ldots,\ell_r\ge 1$.
The latter is an immediate induction on the~$\ell_i$ once we observe that $\cone(s^{\ell+1})$ is an extension of~$\cone(s^{\ell})$ and~$\cone(s)$ and since the case $\ell_1=\ldots=\ell_r=1$ is the hypothesis.
Part~\eqref{it:[ed]-2} then follows from \Cref{Cor:[et]} and the homology long exact sequence for the exact triangle defining the homotopy limit; see~\eqref{eq:holim}.
\end{proof}

\begin{Lem}
\label{Lem:i*e}%
Let $n\ge 1$ and $t\in\Der(R)$. The tt-functor $\ion^*\colon \Der(R)\to \Der(R/\In)$ sends the canonical maps~$\eY\otimes t\to t$ and~$t\to [\eY,t]$ to isomorphisms.
In particular, $\ion^*(\eY)\cong\unit$.
\end{Lem}
\begin{proof}
The cones of these maps are~$\fY\otimes t$ and~$[\fY,t]$ and these objects are killed by the Koszul objects~$\kn$ since $\supp(\kn)=Y$ and since~$\kn$ is dualizable in~$\cT$. So it suffices to observe that~$\ion^*(\kn)\cong\oplus_{i=0}^r \unit^{{n\choose i}}$, so $\unit$ kills~$\ion^*(\fY\otimes t)$ and~$\ion^*([\fY,t])$.
\end{proof}

Recall from \Cref{Rec:Der-Y} the tt-category~$\Der_Y(R)\subseteq\Der(R)$ of complexes supported on~$Y$ and the tt-equivalent category $\Der(R)^\wedge_Y=(\Der_Y(R))^{\perp\perp}\subseteq\Der(R)$ of derived complete complexes along~$Y$.

\begin{Thm}
\label{Thm:dualizable-complete}%
Let $R$ be a commutative ring and let $I\subset R$ be a finitely generated ideal such that $R\cong\hat{R}$ is $I$-adically complete and let $Y=V(I)$.
\begin{enumerate}[\rm(a)]
\item
\label{it:dual-1}%
A complex $d\in\Der_Y(R)$ supported on~$Y$ is dualizable in the category~$\Der_Y(R)$ if and only if $d$ belongs to the thick subcategory generated by the unit~$\eY$.
\smallbreak
\item
\label{it:dual-2}%
A derived complete complex $d\in\Der(R)^{\wedge}_Y$ is dualizable in the category~$\Der(R)^{\wedge}_Y$ if and only if $d$ is a perfect complex.
\end{enumerate}
\end{Thm}
\begin{proof}
We only prove~\eqref{it:dual-1}, as~\eqref{it:dual-2} will follow by the tt-equivalence $\cT_Y\cong\hTY$ of \Cref{Rem:TY=hTY} and the fact that the unit~$\unit_{\hTY}$ is $\hat{\unit}_Y\cong R$ by \Cref{Prop:classical=>derived} under the inclusion~$\hTY=\cT_Y^{\perp\perp}\into\cT=\Der(R)$.
Since dualizable objects always form a thick subcategory containing the unit, we only need to prove that if $d$ is a dualizable object in $\cat T_Y$ then $d \in \thick\langle \eY \rangle$ in $\cat T_Y$.
(Since $\cat T_Y \subset \cat T$ is a thick subcategory, this is the same as $\thick\langle \eY\rangle$ in $\cat T$.)
If~$d=0$ there is nothing to prove, so we could assume that~$d\neq 0$.
We are going to proceed by induction on the `homological amplitude' of the image~$\iota_{1}^*(d)$ of~$d$ in~$\Der(R/I)$ and we will construct an exact triangle in~$\cT_Y$
\begin{equation}
\label{eq:d'PB}%
d'\to \eY\otimes P \xto{g} d \to \Sigma d'
\end{equation}
where $P$ is a finitely-generated projective $R$-module and where either $d'=0$ or the `homological amplitude' of~$\iota_{1}^*(d')$ is less than that of~$d$.
For this `homological amplitude' to make sense, we need to know that $\iota_{1}^*(d)$ is perfect. In fact, we have more generally for the ideals~$\In$ of \Cref{Not:general-completion}:
\begin{equation}
\label{eq:i*(d)}%
\textit{for every $n\ge 1$ the complex $\ion^*(d)\in\Der(R/\In)$ is perfect.}
\end{equation}
To see this, note that the functor~$(\ion^*)\restr{{\cT_Y}}\colon \cT_Y\hook \cT=\Der(R)\to \Der(R/\In)$ sends the tensor product to the tensor product, since both the inclusion $\cT_Y\hook \cT$ and extension-of-scalars~$\ion^*$ do. However, the inclusion $\cT_Y\hook \cT$ does not preserve the unit ($\eY\not\mapsto \unit$). Luckily, the second functor~$\ion^*$ corrects this issue, by \Cref{Lem:i*e}. This proves~\eqref{eq:i*(d)} since a tensor functor preserves dualizable objects and the dualizable objects
in the derived category of a ring are just the perfect complexes.
The `homological amplitude' of~$\iota_{1}^*(d)$ is simply the length of the interval where its homology is non-zero, \ie the smallest number~$b-a$ where $\rmH_i(\iota_{1}^*(d))=0$ for all~$i>b$ and all~$i<a$. By convention, let us say that this amplitude is~$-1$ if all those groups are zero,
that is, when $\iota_{1}^*(d)=0$.
We shall see that this only happens for~$d=0$.

For our dualizable object~$d\in\cT_Y$, define the following object of~$\cT=\Der(R)$:
\[
t\coloneqq [\eY,d].
\]
Note that $[\eY,t]=t$ and $\eY\otimes t\cong d$ and in particular~$t\neq 0$ when~$d\neq 0$. By \Cref{Lem:i*e} again, this object~$t$ has the same image as~$d$ in~$\Der(R/\In)$ under~$\ion^*$:
\begin{equation}
\label{eq:i*t=i*d}%
\ion^*(t)\cong\ion^*(d)
\end{equation}
for every $n\ge1$. In particular, by~\eqref{eq:i*(d)} the object $\ion^*(t)$ is compact in~$\Der(R/\In)$:
\begin{equation}\label{eq:i*t-perfect}%
\qquad
\ion^*(t)\in\Dperf(R/\In)
\quadtext{for all~$n\ge1$.}
\end{equation}

On the other hand, it is a general fact in a tensor-triangulated category~$\cS$ with coproducts and cocontinuous tensor, like~$\cS=\cT_Y$, that if $k$ is compact and $d$ is dualizable then $k\otimes d$ is compact. This is immediate from~$\Homcat{S}(k\otimes d,-)\cong\Homcat{S}(k,d^\vee\otimes -)$.
In particular, in our case, we see that $\kos(s_1,\ldots,s_r)\otimes d$ is compact in~$\cT_Y$, hence in~$\cT$ since $(\cT_Y)^c=(\cT^c)_Y$ by construction of~$\cT_Y$.
We can therefore invoke \Cref{Lem:[ed]} for our~$d$ to conclude that~$t=[\eY,d]$ is right-bounded. Up to replacing~$d$ by a (de)suspension, we can assume that
\begin{equation}
\label{eq:t>0}%
t=[\eY,d]\textrm{ belongs to }\Der_{\ge0}(R),\textrm{ and }\rmH_0(t)\neq 0\textrm{ when }d\neq 0.
\end{equation}

By~\eqref{eq:i*t-perfect} and~\eqref{eq:t>0}, the object~$t$ satisfies the hypotheses of \Cref{Lem:3}, hence
\begin{equation}
\label{eq:H0t}%
\rmH_0(\iota_{1}^*t)=0\quadtext{if and only if}d=0.
\end{equation}
Indeed, the vanishing of~$\rmH_0(\iota_{1}^*t)\cong\rmH_0((R/I)\Lotimes t)$ implies $\rmH_0([\eY,t])=0$ by \Cref{Lem:3}. But $[\eY,t]=t$ and $\rmH_0([\eY,t])=\rmH_0(t)=0$ only happens for~$d=0$ by~\eqref{eq:t>0}. The converse in~\eqref{eq:H0t} is clear since $t$ was defined to be~$[\eY,d]$.

So let us proceed with constructing the exact triangle~\eqref{eq:d'PB}, when $d\neq 0$, that is, when $\rmH_0(t)\neq 0$.
The object~$t\in\Der_{\ge 0}(R)$ is represented by a complex of projective~$R$-modules (not necessarily finitely generated) as in the top row below
\begin{equation}
\label{eq:XXP}%
\vcenter{\xymatrix@R=1em{
t=
&& \cdots \ar[r]
& t_2 \ar[r]
& t_1 \ar[r]
& t_0 \ar[r]
& 0 \ar[r]
& 0 \cdots
\\
\bar{t}=
&& \cdots \ar[r]
& \bar{t}_2 \ar[r]
& \bar{t}_1 \ar[r]
& \bar{t}_0 \ar[r]
& 0 \ar[r]
& 0 \cdots
\\
&&&&& \bar{P} \ar[u]_-{\bar{f}}
}}
\end{equation}
and its image under~$\iota_{1}^*=(R/I)\Lotimes_R-$, displayed in the second row with the shortcut~$\bar{t}_i:=(R/I)\otimes_R t_i$, is a perfect complex by~\eqref{eq:i*t-perfect}. By~\eqref{eq:H0t} we know that $\bar{t}_1\to \bar{t}_0$ cannot be surjective, \ie the `homological amplitude' of~$\iota_{1}^*(d)\cong\iota_{1}^*(t)=\bar{t}$ is the biggest integer $b\ge 0$ such that~$\rmH_b(\bar{t})\neq 0$.
Since $\bar{t}$ is quasi-isomorphic to a bounded complex of finitely generated projective $R/I$-modules, that we can assume to live in non-negative degrees since~$\bar{t}\in\Der_{\ge 0}(R/I)$, there exists a finitely generated projective $R/I$-module~$\bar{P}$ as in the third row of~\eqref{eq:XXP} (simply the last entry of that perfect complex) with the property that the homotopy fiber of $\bar{P}\to \bar{t}$
\begin{equation}
\label{eq:t'Pt-bar}%
\bar{t}' \to \bar{P} \xto{\bar{f}} \bar{t} \to \Sigma\bar{t}'
\end{equation}
is a perfect complex~$\bar{t}'$ with amplitude~$b-1$. (Recall that amplitude~$-1$ means~$\bar{t}'=0$.)

By \Cref{Prop:lift-proj}, since $R$ is classically complete, we know that $\bar{P}\cong (R/I)\otimes_R P$ for some finitely generated projective $R$-module~$P$ and we can lift the map~$\bar{f}\colon \bar{P}\to \bar{t}_0$ to a map~$f\colon P\to t_0$ since $P$ is projective and~$t_0\onto\bar{t}_0$ is a surjection of $R$-modules.
Since $t$ is concentrated in nonnegative degrees, $f$ defines an actual morphism of complexes~$f\colon P\to t$ (in degree zero).
Define~$g:=\eY\otimes f\colon \eY\otimes P\to \eY\otimes t\cong d$ and complete it into an exact triangle as announced in~\eqref{eq:d'PB}. Note that $d'$ is again dualizable in~$\cT_Y$.
So it only remains to show that $\iota_{1}^*(d')$ has homological amplitude~$b-1$.
By \Cref{Lem:i*e} we know that the image of the triangle~\eqref{eq:d'PB} under the tt-functor~$\iota_{1}^*$ is isomorphic to the exact triangle~\eqref{eq:t'Pt-bar}.
In other words, $\iota_{1}^*(d')\simeq \bar{t}'$ which has amplitude one less than that of~$d$.
This finishes the proof.
\end{proof}

\begin{Cor}
\label{Cor:dualizable-complete}%
Suppose that $R\cong\hat{R}_{Y}$ is classically {\Ycomplete}.
Then the subcategory of dualizable objects inside the category $\Der(R)^{\wedge}_Y$ of {\Ycomplete} complexes coincides with that of perfect complexes $\Dperf(R)=\Der(R)^c$ and is tt-equivalent to the subcategory of dualizable objects in~$\Der_Y(R)$ via the following tt-equivalences:
\[
\xymatrix@C=2em@R=2em{
& \Dperf(R) \ar[ld]_-{\eY\otimes-}^-{\cong} \ar@{=}[rd]^-{!}
\\
(\Der_Y(R))^d \ar[rr]_-{\cong}^-{[\eY,-]}
&& (\Der(R)^{\wedge}_Y)^d.
}
\]
\end{Cor}

\begin{proof}
The right-hand equality is \Cref{Thm:dualizable-complete}\,\eqref{it:dual-2}.
The horizontal equivalence is simply the dualizable part of the tt-equivalence $\cT_Y\cong(\cT_Y)^{\perp\perp}=\hTY$ of \Cref{Rem:TY=hTY}.
The left-hand tensor functor $\eY \otimes -:\cat T \to \cat T_Y$ preserves dualizable objects
hence restricts to a functor $\cat T^c=\cat T^d\to (\cat T_Y)^d$.
The composite maps $c\in\cT^c$ to $[\eY,\eY\otimes c]\cong[\eY,\eY]\otimes c=\hat{\unit}_Y\otimes c$ which is simply~$c$ in our case since $\hat{\unit}_Y\cong \unit$ by \Cref{Prop:classical=>derived}.
\end{proof}


\section{Main results}
\label{sec:main}%


We keep our general notation: $R$ is a commutative ring and $Y\subseteq \Spec(R)$ is a closed subset with quasi-compact complement.
We write~$\cT=\Der(R)$ for the derived category, $\cat T_Y=\Der_Y(R)$ for the category of complexes supported on~$Y$ and~$\hTY=\Der(R)^{\wedge}_Y$ for the category of derived complete complexes along~$Y$ (\Cref{Rec:Der-Y}).

\smallskip
We are ready to prove our main result.
\begin{Thm}
\label{Thm:main}%
Let $\underline{s}=(s_1,\ldots,s_r)$ be a sequence of elements of~$R$ such that $Y=V(s_1,\ldots,s_n)$.
Then the following are equivalent:
\begin{enumerate}[\rm(i)]
\item
\label{it:main-1}%
The sequence~$\underline{s}$ is Koszul-complete (\cref{Def:I-kos}).
\smallbreak
\item
\label{it:main-2}%
There is a canonical tt-equivalence $(\DYR)^d \cong \Dperf(\hat{R}_{Y})$ making the following diagram commute:
\[\vcenter{
\xymatrix@R=1em{
& \Dperf(R) \ar@/_1em/[ld]_-{\eY\otimes-} \ar@/^1em/[rd]^-{\hat{R}_{Y}\otimes_R-}
\\
(\Der_Y(R))^d \ar[rr]^-{\cong}
&& \Dperf(\hat{R}_{Y}).
}}
\]
\item
\label{it:main-3}%
There is a canonical tt-equivalence $(\Der(R)^{\wedge}_{Y})^d \cong \Dperf(\hat{R}_{Y})$ making the following diagram commute:
\[\vcenter{
\xymatrix@R=1em{
& \Dperf(R) \ar@/_1em/[ld]_-{(-)^\wedge_Y} \ar@/^1em/[rd]^-{\hat{R}_{Y}\otimes_R-}
\\
(\Der(R)^{\wedge}_Y)^d \ar[rr]^-{\cong}
&& \Dperf(\hat{R}_{Y}).
}}
\]
\smallbreak
\item
\label{it:main-4}%
There exists an isomorphism of ring objects~$\hat{\unit}_Y\cong\hat{R}_{Y}$ in~$\Der(R)$.
\end{enumerate}
\end{Thm}
\begin{proof}
Let $f:\Spec(\hat{R}_{{}})\to\Spec(R)$ be the ring-theoretic completion
and let $f^*:\cat T\coloneqq \Der(R) \to \Der(\hat{R}_{{}}) \eqqcolon \cat S$ be extension-of-scalars.
Write $Y'=f\inv(Y)$ as usual.

\eqref{it:main-1}$\then$\eqref{it:main-2}:
We have a commutative diagram of tt-functors
\[
\xymatrix@C=4em{
\cat T \ar[r]^-{f^*} \ar[d]_-{\eY\otimes-}
& \cat S \ar[d]^-{\ee_{Y'}\otimes-}
\\
\cat T_Y \ar[r]^-{f^*\restr{\cT_Y}}
& \cat S_{Y'}.
}\]
The bottom functor is an equivalence by \cref{Thm:D(R)-completion}. On the other hand, since $\hat{R}_{{}}$ is classically $Y'$-complete (\cref{Rem:fg-comp-is-comp}),  \cref{Cor:dualizable-complete} implies that the right vertical functor is an equivalence when restricted to dualizable objects. We get
\[
\xymatrix{
\cat T^c \ar[r]^-{f^*} \ar[d]_-{\eY\otimes-}
& \cat S^c \ar[d]^-{\cong}
\\
(\cat T_Y)^d \ar[r]^-{\cong}
& (\cat S_{f^{-1}(Y)})^d.
}\]

\eqref{it:main-2}$\then$\eqref{it:main-1}:
The hypothesis is that we have a tt-equivalence making the right-hand triangle below commute.
\[
\xymatrix@C=4em{
& \cat T^c \ar[d]_-{\eY\otimes-} \ar@/^1em/[dr]^-{f^*|_{\cat T^c}}
\\
\cat T_Y^c \ \vphantom{I^I}\ar@{^(->}[r]^-{\incl} \ar@/^1em/@{^(->}[ur]^-{\incl}
& (\cat T_Y)^d \ar[r]^-{\cong} & \cat S^c 	
}
\]
The left-hand triangle commutes since~$\eY\otimes-$ is right adjoint to the inclusion~$\cT_Y\into\cT$ and in particular retracts it.
The bottom composite is fully faithful, hence for any $a,b \in \cat T_Y^c$, we have that $\cat T(a,b) \to \cat S(f^*(a),f^*(b)) \cong \cat T(a,f_*(\unit)\otimes b)$ is a bijection.
Since this is true for all $a \in \cat T_Y^c$ it is also true for all $a \in \Loc(\cat T_Y^c)=\cat T_Y$ and the latter contains~$b$ and~$f_*(\unit)\otimes b$.
It follows by Yoneda that $\eta_b\colon b \to f_*(\unit)\otimes b$ is an isomorphism for every $b \in \cat T_Y^c$.
Plugging in $b=\kos(\underline{s})$ we get~\eqref{it:main-1}.

The equivalence~\eqref{it:main-1}$\Leftrightarrow$\eqref{it:main-4} was already established in \Cref{Thm:D(R)-completion}.
The equivalence~\eqref{it:main-2}$\Leftrightarrow$\eqref{it:main-3} follows again from the tt-equivalence~$\cT_Y\cong\cT^\wedge_Y$ of \Cref{Rem:TY=hTY}.
\end{proof}

\begin{Cor}
\label{Cor:main-noeth}%
Let $R$ be a noetherian commutative ring and let $Y\subseteq\Spec(R)$ be a closed subset. There are canonical equivalences of tt-categories between $\smash{\Dperf(\hat{R}_{{}})}$ and the subcategories~$(\DYR)^d$ and~$(\Der(R)^\wedge_Y)^d$ making the following diagram commute
\[\vcenter{
\xymatrix@R=1em{
&& \Dperf(R) \ar@/_1em/[ld]^-{(-)^\wedge_Y} \ar@/_1em/[lld]_-{\eY\otimes-} \ar@/^1em/[rd]^-{\hat{R}_{{}}\otimes_R-}
\\
(\DYR)^d \kern-.5em \ar@{}[r]|-{\displaystyle\cong}
& \kern-.5em (\Der(R)^\wedge_Y)^d \ar[rr]^-{\cong}
&& \Dperf(\hat{R}_{{}}).
}}
\]
\end{Cor}
\begin{proof}
Immediate from \Cref{Thm:main} and \Cref{prop:noeth-koszul} (and \Cref{Rem:TY=hTY}).
\end{proof}

\begin{Rem}
\label{Rem:no}%
In view of \Cref{Thm:main} and \Cref{Exa:no} the equivalence of \Cref{Cor:main-noeth} does not hold for the non-noetherian ring $R=\bbZ_{(p)}\oplus (\bbQ/\bbZ_{(p)})$ and $Y=V(p)$.
\end{Rem}

\begin{Rem}
\label{Rem:final}%
We learned in \Cref{Exa:trivial-Der} that $\hTY$ should not be thought of as the tt-analogue of ring-completion. \Cref{Thm:tt-compl} and \Cref{Thm:D(R)-completion} suggest that~$\hTY$ only recovers the part of {\Ycompletion} \emph{supported on~$Y$}. Perhaps one should write~$\cT^{\wedge}_{YY}$ to make this point: One decoration~${}^\wedge_Y$ for {\Ycompletion} and another~$Y$ for support.

We could also say that the only thing we need to build the full completion~$\hat{\cT}$ is a rigid tt-category~$\hat{\cT}^d$ that will serve as the dualizable objects in~$\hat\cT$ and that we can Ind-complete into~$\hat\cT$. \Cref{Thm:main} suggests that the dualizable objects in~$\hTY$ is a possible choice for this~$\hat\cT^c$.
This is the definition chosen in~\cite{NaumannPolRamzi24}.
Another, possibly smaller, choice would be to take the smallest tt-subcategory of~$(\hTY)^d$ that contains the image of~$\cT^c$.
It is not clear if there is a difference in general between these two choices and
\Cref{Thm:dualizable-complete} tells us that there is none in the case of~$\Der(R)$.
Future investigation of this topic seems worth pursuing in general tt-geometry.
\end{Rem}


\begin{thebibliography}{BIKP23}

\bibitem[BDS16]{BalmerDellAmbrogioSanders16}
Paul Balmer, Ivo Dell'Ambrogio, and Beren Sanders.
\newblock Grothendieck--{N}eeman duality and the {W}irthm\"uller isomorphism.
\newblock {\em Compos. Math.}, 152(8):1740--1776, 2016.

\bibitem[BF11]{BalmerFavi11}
Paul Balmer and Giordano Favi.
\newblock Generalized tensor idempotents and the telescope conjecture.
\newblock {\em Proc. Lond. Math. Soc. (3)}, 102(6):1161--1185, 2011.

\bibitem[BIKP23]{BensonIyengarKrausePevtsova23}
Dave Benson, Srikanth~B. Iyengar, Henning Krause, and Julia Pevtsova.
\newblock Local dualisable objects in local algebra.
\newblock Preprint, 15~pages, available online at
  \href{https://arxiv.org/abs/2302.08562}{arXiv:2302.08562}, 2023.

\bibitem[DG02]{DwyerGreenlees02}
W.~G. Dwyer and J.~P.~C. Greenlees.
\newblock Complete modules and torsion modules.
\newblock {\em Amer. J. Math.}, 124(1):199--220, 2002.

\bibitem[GM92]{GreenleesMay92}
J.~P.~C. Greenlees and J.~P. May.
\newblock Derived functors of {$I$}-adic completion and local homology.
\newblock {\em J. Algebra}, 149(2):438--453, 1992.

\bibitem[Gre01]{Greenlees01}
J.~P.~C. Greenlees.
\newblock Tate cohomology in axiomatic stable homotopy theory.
\newblock In {\em Cohomological methods in homotopy theory ({B}ellaterra,
  1998)}, volume 196 of {\em Progr. Math.}, pages 149--176. Birkh\"auser,
  Basel, 2001.

\bibitem[HPS97]{HoveyPalmieriStrickland97}
Mark Hovey, John~H. Palmieri, and Neil~P. Strickland.
\newblock Axiomatic stable homotopy theory.
\newblock {\em Mem. Amer. Math. Soc.}, 128(610), 1997.

\bibitem[Ked24]{prismatic-online}
Kiran Kedlaya.
\newblock Notes on prismatic cohomology.
\newblock \url{https://kskedlaya.org/prismatic/prismatic.html}, 2024.

\bibitem[Nee92]{Neeman92b}
Amnon Neeman.
\newblock The connection between the {$K$}-theory localization theorem of
  {T}homason, {T}robaugh and {Y}ao and the smashing subcategories of
  {B}ousfield and {R}avenel.
\newblock {\em Ann. Sci. \'Ecole Norm. Sup. (4)}, 25(5):547--566, 1992.

\bibitem[Nee96]{Neeman96}
Amnon Neeman.
\newblock The {G}rothendieck duality theorem via {B}ousfield's techniques and
  {B}rown representability.
\newblock {\em J. Amer. Math. Soc.}, 9(1):205--236, 1996.

\bibitem[Nee01]{Neeman01}
Amnon Neeman.
\newblock {\em Triangulated categories}, volume 148 of {\em Annals of
  Mathematics Studies}.
\newblock Princeton University Press, 2001.

\bibitem[NPR24]{NaumannPolRamzi24}
Niko Naumann, Luca Pol, and Maxime Ramzi.
\newblock A symmetric monoidal fracture square, 2024.
\newblock Preprint, 40~pages, available online at
  \href{https://arxiv.org/abs/2411.05467}{arXiv:2411.05467}.

\bibitem[Pos23]{Positselski23}
Leonid Positselski.
\newblock Remarks on derived complete modules and complexes.
\newblock {\em Math. Nachr.}, 296(2):811--839, 2023.

\bibitem[PSY14]{PortaShaulYekutieli14}
Marco Porta, Liran Shaul, and Amnon Yekutieli.
\newblock On the homology of completion and torsion.
\newblock {\em Algebr. Represent. Theory}, 17(1):31--67, 2014.

\bibitem[Sch03]{Schenzel03}
Peter Schenzel.
\newblock Proregular sequences, local cohomology, and completion.
\newblock {\em Math. Scand.}, 92(2):161--180, 2003.

\bibitem[{Sta}20]{stacks-project}
The {Stacks Project Authors}.
\newblock {\it {S}tacks {P}roject}.
\newblock \url{http://stacks.math.columbia.edu}, 2020.

\bibitem[Ste13]{StevensonApp13}
Greg Stevenson.
\newblock Disconnecting spectra via localization sequences.
\newblock {\em J. K-Theory}, 11(2):324--329, 2013.
\newblock Appendix to: \textit{Module categories for group algebras over
  commutative rings}, by D.\ Benson, H.\ Krause and S.\ Iyengar.

\bibitem[Tho97]{Thomason97}
R.~W. Thomason.
\newblock The classification of triangulated subcategories.
\newblock {\em Compositio Math.}, 105(1):1--27, 1997.

\bibitem[Yek11]{Yekutieli11}
Amnon Yekutieli.
\newblock On flatness and completion for infinitely generated modules over
  {N}oetherian rings.
\newblock {\em Comm. Algebra}, 39(11):4221--4245, 2011.

\end{thebibliography}

\end{document}